\documentclass[11pt]{amsart}

\usepackage{graphicx}
\usepackage{color}
\usepackage{rotating}
\usepackage{rotfloat}
\usepackage{geometry}
\usepackage[all]{xy}
\usepackage{array}
\usepackage{tkz-graph}
\usetikzlibrary{arrows}

\newcolumntype{M}{>{\centering\arraybackslash}m{\dimexpr.22\linewidth-2\tabcolsep}}

\usepackage{mathrsfs}
\usepackage{bbold}
\usepackage[T1]{fontenc} 
\usepackage{textcomp}
\usepackage{times}
\usepackage[scaled=0.92]{helvet} 

\font\eulersm=eusm10 at 11pt
\def\esm#1{\hbox{\eulersm {#1}}}
\newcommand{\T}{\esm{T}}

\font\eulersmm=eusm10 at 9pt
\def\esmm#1{\hbox{\eulersmm {#1}}}
\newcommand{\sT}{\esmm{T}}

\newtheorem{introtheorem}{Theorem}

\newtheorem{introcorollary}[introtheorem]{Corollary}
\theoremstyle{plain}
\swapnumbers
\newtheorem{theorem}{Theorem}[section]
\newtheorem{proposition}[theorem]{Proposition}
\newtheorem{lemma}[theorem]{Lemma}
\newtheorem{corollary}[theorem]{Corollary}
\theoremstyle{definition}
\newtheorem{definition}[theorem]{Definition}
\newtheorem{example}[theorem]{Example}
\newtheorem{remark}[theorem]{Remark}
\newtheorem{se}[theorem]{}

\newcommand{\pushright}[1]{\ifmeasuring@#1\else\omit\hfill$\displaystyle#1$\fi\ignorespaces}

\usepackage{amsfonts}
\usepackage{amssymb}

\def\til#1{\widetilde{#1}}
\def\ovl#1{\overline{#1}}

\renewcommand{\tilde}{\widetilde}
\renewcommand{\bar}{\overline}
\newcommand{\isomto}{\overset{\sim}{\rightarrow}}
\DeclareMathOperator{\PP}{\mathbf{P}}

\DeclareMathOperator{\Q}{\mathbf{Q}}
\DeclareMathOperator{\R}{\mathbf{R}}
\DeclareMathOperator{\C}{\mathbf{C}}
\DeclareMathOperator{\F}{\mathbf{F}}
\DeclareMathOperator{\p}{\mathfrak{p}}
\DeclareMathOperator{\GL}{\mathrm{GL}}
\DeclareMathOperator{\PGL}{\mathrm{PGL}}

\DeclareMathOperator{\fn}{\mathfrak{n}}

\DeclareMathOperator{\vol}{\mathrm{vol}}

\DeclareMathOperator{\gon}{gon}
\DeclareMathOperator{\sgon}{sgon}

\DeclareMathOperator{\V}{V}
\DeclareMathOperator{\E}{E}

\newcommand{\f}{\varphi}
\newcommand{\inv}{^{-1}}
\newcommand{\nn}{\nonumber}

\newcommand\size{\mathrm{size}}
\def\quotient#1#2{%
    \raise1ex\hbox{$#1$}\Big/\lower1ex\hbox{$#2$}%
}

\begin{document}

\date{\today\ (version 2.0)} 
\title[A combinatorial Li--Yau inequality and rational points on curves]{A combinatorial Li--Yau inequality \\ and rational points on curves}
\author[G.~Cornelissen]{Gunther Cornelissen}
\address{\normalfont Mathematisch Instituut, Universiteit Utrecht, Postbus 80.010, 3508 TA Utrecht, Nederland}
\email{g.cornelissen@uu.nl}
\author[F.~Kato]{Fumiharu Kato}
\address{\normalfont Department of Mathematics, Kumamoto University, Kurokami Kumamoto 860-8555, Japan} 
\email{kato@sci.kumamoto-u.ac.jp}
\author[J.~Kool]{Janne Kool}
\address{\normalfont Mathematisch Instituut, Universiteit Utrecht, Postbus 80.010, 3508 TA Utrecht, Nederland}
\curraddr{\normalfont Max-Planck-Institut f\"ur Mathematik, Postfach 7280, 53072 Bonn, Deutschland}
\email{jannekool@gmail.com}


\begin{abstract} \noindent We present a method to control gonality of nonarchimedean curves based on graph theory. 

Let $k$ denote a complete nonarchimedean valued field. We first prove a lower bound for the gonality of a curve over the algebraic closure of $k$ in terms of the minimal degree of a class of graph maps, namely: one should minimize over all so-called finite harmonic graph morphisms to trees, that originate from any \emph{refinement} of the dual graph of the stable model of the curve. 

Next comes our main result: we prove a lower bound for the degree of such a graph morphism in terms of the first eigenvalue of the Laplacian and some ``volume'' of the original graph; this can be seen as a substitute for graphs of the Li--Yau inequality from differential geometry, although we also prove that the strict analogue of the original inequality fails for general graphs. 

Finally, we apply the results to give a lower bound for the gonality of arbitrary Drinfeld modular curves over finite fields and for general congruence subgroups $\Gamma$ of $\Gamma(1)$ that is linear in the index $[\Gamma(1):\Gamma]$, with a constant that only depends on the residue field degree and the degree of the chosen ``infinite'' place. This is a function field analogue of a theorem of Abramovich for classical modular curves. We present applications to uniform boundedness of torsion of rank two Drinfeld modules that improve upon existing results, and to lower bounds on the modular degree of certain elliptic curves over function fields that solve a problem of Papikian.  
\end{abstract}

\thanks{We thank Omid Amini, Matthew Baker, Dion Gijswijt, Mihran Papikian and Andreas Schweizer for very useful comments on a previous version of the manuscript. We also thank the authors of \cite{Amini} and \cite{Cadoret} for providing us with a preliminary version of their manuscripts.} 
\maketitle

\tableofcontents 

\section*{Introduction}
The \emph{gonality} $\mathrm{gon}_{{k}}(X)$ of a smooth projective curve $X$ over a field $k$ is defined as the minimal degree of a non-constant morphism from $X$ to the projective line $\PP^1_k$. If $k=\mathbf{C}$ is the complex numbers, $X$ can be considered as a compact Riemann surface, and Li and Yau \cite{LiYau} have established a lower bound on the gonality of $X$ over $\mathbf{C}$ in terms of of the hyperbolic volume and the first eigenvalue of the Laplacian of $X$. Such a bound has numerous applications, of which we mention one: 
Abramovich \cite{Abramovich} has combined it with a lower bound on the eigenvalue arising from the theory of automorphic forms (of which the currently sharpest value was given by Kim and Sarnak \cite{Kim}) to prove a lower bound on the gonality of modular curves for congruence groups that is linear in the genus of the curves (or, what is the same, linear in the index of the group in the full modular group). In this paper, we study a nonarchimedean analogue of these results. 

\medskip

The first result is an inequality between the (geometric) gonality $\mathrm{gon}_{\bar{k}}(X)$ of a curve $X$ defined over a complete nonarchimedean valued field $k$ and the ``gonality'' of the reduction graphs of suitable models of the curve. There are various complications, such as to establish a good theory for the reduction of a covering map $X \rightarrow \PP^1$. Such a map extends to the stable model, but not necessarily as a \emph{finite} morphism. This can be remedied by choosing suitable semi-stable models. The problem was studied by Liu and Lorenzini \cite{LiuLorenzini}, Coleman \cite{Coleman} and Liu \cite{Liu}, and more recently in \cite{Amini}. In Section \ref{extension}, we provide another (similar) solution, directly adapted to the applications that we have in mind. 

Next, we relate the gonality of the special fiber to what we call the \emph{stable gonality} of the intersection dual graph. For standard graph terminology, we refer to Section \ref{graphs}.  We also need the notion of an (indexed) finite harmonic graph morphism, for which we refer to Definition \ref{defmor}.  Given a graph $G$, then another graph $G'$ is called a \emph{refinement} of $G$ if it can be obtained from $G$ by finitely often performing the two following operations: (a) subdivision of an edge; (b) addition of a \emph{leaf}, i.e., the addition of an extra vertex and an edge between this vertex and a vertex of the already existing graph. The \emph{stable gonality} of $G$, denoted 
$ \mathrm{sgon}(G)$, is defined as the minimal degree of a finite harmonic morphism from any refinement of $G$ to a tree. This relates to, but is different from previous notions of gonality for graphs as introduced by Baker and Norine \cite{BakerNorine}, and Caporaso \cite{Caporaso} (cf.\ Appendix \ref{app} for a discussion of these different notions and how they relate to stable gonality). 

\begin{introtheorem}[= Corollary \ref{cor-31}] \label{compare}
Let $X$ be a geometrically connected projective smooth curve over a complete nonarchimedean valued $k$ with valuation ring $R$, and $\mathscr{X}$ the stable $R$-model of $X$.
Let $\ovl{k}$ be an algebraic closure of $k$. Let $\Delta(\mathscr{X}_0)$ denote the intersection dual graph of the special fiber $\mathscr{X}_0$. 
Then we have
$$
\mathrm{gon}_{\bar{k}}(X)\geq\mathrm{sgon}(\Delta(\mathscr{X}_0)). 
$$
\end{introtheorem}

Two examples (\ref{examplegon1} and \ref{examplegon2})  illustrate that both refinements operation are necessary. First, the ``banana graph'' $B_n$ given by two vertices joined by $n>1$ distinct edges has stable gonality $2$, although the minimal degree of a finite harmonic graph morphism from $B_n$ itself (without any refinement) to a tree is $n$.  Secondly, the minimal degree of a finite harmonic graph morphism from any subdivision of the complete graph $K_4$ to a tree is $4$. However, by adding leaves, the stable gonality can be shown to be $3$.  

We then prove an analogue for graphs of the upper bound on gonality from Brill-Noether theory for the gonality of curves over arbitrary fields (in this generality a theorem of Kleiman--Laksov \cite{Kleiman}):
 
\begin{introtheorem}[=Theorem \ref{BNN2}] \label{BNN}
For any graph $G$ with first Betti number $g \geq 2$, we have an upper bound $$\mathrm{sgon}(G)\leq\lfloor\frac{g+3}{2}\rfloor.$$
\end{introtheorem}

The main result is a spectral lower bound for the stable gonality of a graph. 
Let $\lambda_G$ denote the first non-trivial (i.e., smallest non-zero) eigenvalue of the Laplacian $L_G$ of $G$, and let $$\Delta_G:=\max\{\deg(v) \colon v\in V(G)\}$$ denote the maximal vertex degree of $G$. Finally, let $|G|$ denote the number of vertices of $G$. Then we have

\begin{introtheorem}[= Corollary \ref{deafschatting}] \label{spectral} The stable gonality of a graph $G$ satisfies \[\sgon(G)\geq \left\lceil \frac{\lambda_G}{\lambda_G+4(\Delta_G+1)}|G| \right \rceil.\] \end{introtheorem}

An attractive feature of the formula is that the lower bound depends on spectral data for the original graph, not on all possible refinements of the graph. Also, in the bound, one may replace $(\lambda_G, \Delta_G, |G|)$  by the corresponding data $(\lambda_{G'}, \Delta_{G'}, |G'|)$ of any refinement $G'$ of the graph $G$.  

A similar result can be proven using the normalized graph Laplacian, replacing $|G|$ by the ``volume'' of the graph, cf.\ Theorem \ref{genormaliseerdegrens}.

The result can be seen as an analogue of the Li--Yau inequality in differential geometry \cite{LiYau}, which states that the gonality $\mathrm{gon}(X)$ of a compact Riemann surface $X$ (minimal degree of a conformal mapping from $X$ to the Riemann sphere) is bounded below by
$$ \mathrm{gon}(X) \geq \frac{1}{8 \pi} \lambda_X \mathrm{vol}(X), $$
where $\lambda_X$ is the first non-trivial eigenvalue of the Laplace-Beltrami operator of $X$, and $\mathrm{vol}(X)$ denotes the volume of $X$.  In Remark \ref{LYfout}, we will show that the strict graph theory analog of such a formula fails.

\medskip 

We then apply the two theorems above to Drinfeld modular curves over a general global function field $K$ over a finite field with $q$ elements, and we find the positive characteristic analogue of Abramovich's result. In the applications, we will write $|\fn|_\infty$ for the valuation corresponding to a fixed ``infinite'' place $\infty$ of degree $\delta$ of $K$, we denote by 
 $A$ the subring of $K$ of elements that are regular outside $\infty$, and we let $Y$ denote a rank-two $A$-lattice in the completion $K_\infty$ of $K$ at $\infty$. Up to equivalence, such lattices correspond to elements of $\mathrm{Pic}(A)$. Let $H$ denote the maximal abelian extension of $K$ inside $k=K_\infty$; then $\mathrm{Gal}(H/K) \cong \mathrm{Pic}(A)$. 
 In the ``standard'' example where $K=\F_q(T)$ is the function field of $\PP^1$ and $\infty=T^{-1}$, $Y=A \oplus A$ is unique up to equivalence, and $H=K$. Congruence subgroups $\Gamma$ of $\Gamma(Y):=\GL(Y)$ (i.e., containing $\ker \left( \Gamma(Y) \rightarrow \GL(Y/\fn Y) \right)$ for some non-trivial ideal $\fn$ of $A$) act by fractional linear transformations on the Drinfeld ``upper half plane'' $\Omega$, and the quotient analytic space can be compactified into a smooth projective curve $X_\Gamma$ by adding finitely many cusps.

\begin{introtheorem}[= Theorem \ref{DrinGon2}] \label{DrinGon} 
Let $\Gamma$ denote a congruence subgroup of $\Gamma(Y)$. Then the gonality of the corresponding Drinfeld modular curve $X_\Gamma$ satisfies 
$$ \mathrm{gon}_{\bar{K}}(X_{\Gamma}) \geq c_{q,\delta} \cdot [\Gamma(Y):\Gamma] $$
where the constant $c_{q,\delta}$ is 
\[c_{q,\delta}:= \frac{q^\delta-2\sqrt{q^\delta}}{5q^\delta-2\sqrt{q^\delta}+8}\cdot\frac{1}{q(q^2-1)}\]
This implies a linear lower bound in the genus of modular curves of the form
$$ \mathrm{gon}_{\bar{K}}(X_{\Gamma}) \geq c'_{K,\delta} \cdot (g(X_\Gamma) -1 ),$$
where $c'_{K,\delta}$ is a bound that depends only on the function field $K$ and the degree $\delta$ of $\infty$. If $K$ is a rational function field and $\delta=1$, then we can put $c'_{K,\delta}=2c_{q,1}$. 
\end{introtheorem}

In the proof, we use the structure of the reduction graph of the principal modular curve of level $\fn$ (or rather, its components $X(Y,\fn)$ indexed by $Y$ running through classes in $\mathrm{Pic}(A)$). Also used in the proof is a bound of the Laplace eigenvalue for this graph that follows from the Ramanujan conjecture, proven by Drinfeld (in combination with the Courant-Weyl inequalities). The proof of the genus bound is not entirely automatic, due to possible wild ramification. The constant $c_{q,\delta}$ is probably not optimal, and it would be interesting to know whether it can be replaced by an absolute constant, or at least a constant depending on $q$, but tending to an absolute non-zero constant as $q$ increases. Also notice that the bound is vacuous (since a negative number) if $q^\delta<4$, and that the general upper bound $(g+3)/2$ implies that any suitable constant $c'_{K,\delta}$ should be smaller than $5/2$. 

All previously known results on gonality of Drinfeld modular curves used point counting arguments modulo primes, rather than the above ``geometric analysis'' method. 
The best previously known bounds, due to Andreas Schweizer (\cite{SchweizerForum}, Thm.\ 2.4) are not linear in the index and are established for a rational function field $K=\F_q(T)$ only. 
An extra asset of the new method is that it works without much extra effort for a general function field $K$, rather than just a rational function field. 

\medskip

A first application arises from the modularity of elliptic curves over function fields. Recall that any elliptic curve $E/K$ with split multiplicative reduction at the infinite place $\infty$ admits a modular parametrization $\phi \colon X_0(Y,\fn) \rightarrow E$ \cite{GekelerReversat} for some suitable Drinfeld modular curve $X_0(Y,\fn)$. This parametrization is defined over $H$.

\begin{introtheorem}[= Corollary \ref{ModDeg2}] \label{ModDeg}
Let $E/K$ denote an elliptic curve with split multiplicative reduction at the place $\infty$, of conductor $\fn \cdot \infty$. Then the degree of a modular parametrization $\phi \colon X_0(Y,\fn) \rightarrow E$ is bounded below by 
$$\deg \phi \geq \frac{1}{2} c_{q,\delta} [\Gamma(1):\Gamma_0(Y,\fn)]. $$
In particular, we have 
$$ \deg \phi \gg_{q,\delta} |\fn|_\infty. $$
\end{introtheorem}

As usual, $X \gg_y Z$ means that there exists a constant $C_y$ depending only on $y$ such that $X \geq C_y Z$.
 
For $K=\F_q(T)$ a rational function field, the final statement of the theorem confirms a conjecture of Papikian \cite{Papikian}, who had proven (using Spziro's conjecture for function fields and estimating symmetric square $L$-functions by the Ramanujan conjecture) that 
$$ \deg_{ns}(j_E) \cdot \deg \phi \gg_{q,\varepsilon}  |\fn|_\infty^{1-\varepsilon},$$ where $j_E$ is the $j$-invariant of $E$ and $\deg_{ns}(j_E)$ is its inseparability degree. Actually, since he and P\'al have also proven an upper bound 
we conclude that if $E$ is a \emph{strong} Weil curve over $\F_q(T)$ with square-free conductor, then 
$$  |\fn|_\infty \ll_{q} \deg \phi \ll_{q,\varepsilon} |\fn|_\infty^{2+\varepsilon}$$
for any $\varepsilon>0$; cf.\ Remark \ref{pal} for a more precise upper bound, and a discussion of the r\^ole of the Manin constant of $E$.  Contrary to the case of elliptic curves over $\Q$, Gekeler has proven that the modular degree always equals the congruence number of the associated automorphic form \cite{GekelerJTNB} \cite{Cojocaru}.  Hence these results also hold for the congruence number. 

\medskip

We then give applications to rational points of bounded degree on general curves over function fields. In its general form, the theorem gives a finiteness result for points whose degree is bounded  above by ``spectral'' data associated to the combinatorics of a special fiber: 

\begin{introtheorem}[=Theorem \ref{genth2}] \label{genth} Let $X$ denote a curve over a global function field $K$,
such that its Jacobian does not admit a $\bar{K}$-morphism to a curve defined over a finite field. Let $K_\infty$ denote the completion of $K$ at a place $\infty$, and let $G$ denote the stable reduction graph of $X/K_\infty$. Let $\Delta$ denote the maximal vertex degree of $G$, $|G|$ the number of vertices of $G$ and $\lambda$ the smallest non-zero eigenvalue of the Laplacian of $G$. Then the set 
$$  \bigcup_{[K':K] \leq \frac{\lambda(|G|-1)-4\Delta-4}{2\lambda+8\Delta+8}} X(K') $$
of rational points on $X$ of degree at most $\frac{\lambda(|G|-1)-4\Delta-4}{2\lambda+8\Delta+8}$ is finite.
\end{introtheorem}

This applies in particular to various modular curves, as follows.

\begin{introtheorem}[= Theorem \ref{Tors2}] \label{Tors} With the same notations as in Theorem \ref{DrinGon}, if  $X_\Gamma$ is defined over a finite extension $K_\Gamma$ of $K$, then the set 
$$  \bigcup_{[L:K_\Gamma] \leq  \frac{1}{2} \left( c_{q,\delta} \cdot [\Gamma(1):\Gamma]-1\right)} X_\Gamma(L) $$
is finite. 
\end{introtheorem}

Applications to uniform bounds on isogenies and torsion points follow by applying an analogue of a method of Abramovich and Harris \cite{AH} and Frey \cite{Frey}. Recall that $H$ is the maximal abelian extension of $K$ inside $K_\infty$. 

\begin{introcorollary}[= Corollary \ref{Freyan2}] \label{Freyan}
If $\p$ is a prime ideal in $A$, then the set of all rank two Drinfeld $A$-modules defined over some field extension $L$ of $K$ that satisfies the degree bound $$[LH:H] \leq \frac{1}{2} c_{q,\delta} \cdot |\p|_\infty$$ that admit an $L$-rational $\p$-isogeny is finite.  
\end{introcorollary}

We also deduce the following analogue of a result of Kamieny and Mazur \cite{KM}: 

\begin{introcorollary}[= Corollary \ref{KM2}] \label{KM}
Fix a prime $\p$ of $A$. There is a uniform bound on the size of the $\p$-primary torsion of any rank two $A$-Drinfeld module over $L$, where $L$ ranges over all extensions for which the degree $[LH:H]$ is bounded by a given constant. 
\end{introcorollary}

This implies that the uniform boundedness conjecture for rank-two $A$-Drinfeld modules over $K$ follows from the following statement: for fixed $d$,  there are only finitely many $\p$ such that there exists an $L$-rational $\p$-torsion point on an $A$-Drinfeld module over $L$ with $[L:K] \leq d$.

For a rational function field $K=\F_q(T)$, the above two corollaries were proven by Schweizer \cite{SchweizerMZ}. 
The finite bound from these two results is not effective in the \emph{number} of rational points. For effective results on the number of points of low degree on some Drinfeld modular curves, see for example Armana \cite{Armana}. No analogue of Merel's theorem (uniform boundedness of torsion) is currently known for rank-two Drinfeld modules (compare also Poonen \cite{PoonenDrinfeld}). 

As a final remark, there has recently been a surge in the use of gonality and graph theory in arithmetic, but mainly in characteristic zero; for example in the work of Ellenberg, Hall and Kowalski on generic large Galois image, coupling gonality to expander properties of Cayley graphs embedded in Riemann surfaces \cite{Ellenberg}. 
Also in our applications, in a rather different way, the graph expansion properties of the reduction graphs of Drinfeld modular curves seem to intervene in a crucial way in establishing interesting lower bound for their gonality (originally, over rational function fields, we deduced the bounds from natural bounded concentrator properties of subgraphs, as in the work of Morgenstern \cite{Morgenstern}). More generally, the stable gonality in a family of Ramanujan graphs with fixed regularity is bounded below linearly in the number of vertices (cf.\ Remark \ref{rama}). 

\section{Extension of covering maps} \label{extension}
\begin{se}\label{ntn-0}
Let $k$ denote a complete nonarchimedean valued field with valuation ring $R$ with uniformizer $\pi$ and residue field of characteristic $p \geq 0$.
For any $R$-scheme $\mathscr{X}$ we denote by $\mathscr{X}_{\eta}$ (respectively $\mathscr{X}_0$) the generic fiber (respectively the closed fiber). We denote by $\mathscr{X}_{\mathrm{sing}}$ the singular locus of $\mathscr{X}$. 
\end{se}

\begin{se} 
Let $X$ be a geometrically connected projective smooth curve over $k$.
An {\em $R$-model} of $X$ is a pair $\mathscr{X}=(\mathscr{X},\phi)$ consisting of an integral normal scheme $\mathscr{X}$ that is projective and flat over $R$ and a $k$-isomorphism $\phi\colon\mathscr{X}_{\eta}\stackrel{\sim}{\rightarrow}X$.
An $R$-model $\mathscr{X}$ of $X$ is said to be {\em semi-stable} if its special fiber $\mathscr{X}_0$ is reduced with only ordinary double points as singularities. Such a model is called  \emph{stable} if any irreducible component of the special fiber has a finite automorphism group as a marked curve,  where the marking is given by its intersection points with other components. 
\end{se}

\begin{se}
As was shown by Liu and Lorenzini in \cite{LiuLorenzini}, every finite morphism $f\colon X\rightarrow Y$  between geometrically connected projective smooth curves over $k$ extends to a morphism between the stable models of $X$ and $Y$, but the resulting map is \emph{not necessarily finite}. (Similar problems were already encountered and studied by Abhyankar in \cite{ab}.) However, there exists a semi-stable model admitting an  extension of the map that is a finite morphism, as was shown by Coleman \cite{Coleman} and Liu \cite{Liu}. We need a slightly different statement, that we prove along similar lines as Liu:  
\end{se}

\begin{theorem}\label{theorem-21}
Let $f\colon X\rightarrow Y$ be a finite morphism between geometrically connected projective smooth curves over $k$, and $\mathscr{X}$ an $R$-model of $X$.
Then there exist a finite separable field extension $k'/k$, semi-stable $R'$-models $\mathscr{X}'$ and $\mathscr{Y}'$ of $X_{k'}$ and $Y_{k'}$, respectively, over  the integral closure $R'$ of $R$ in $k'$, and an $R'$-morphism $\varphi\colon\mathscr{X}'\rightarrow\mathscr{Y}'$ such that the following conditions are satisfied$:$
\begin{itemize}
\item[{\rm (a)}] $\mathscr{X}'$ dominates $\mathscr{X}_{R'};$
\item[{\rm (b)}] $\varphi$ is finite, surjective, and extends $f_{k'};$
\item[{\rm (c)}] the induced morphism $\varphi_0\colon\mathscr{X}'_0\rightarrow\mathscr{Y}'_0$ satisfies $\varphi^{-1}_0((\mathscr{Y}'_0)_{\mathrm{sing}})=(\mathscr{X}'_0)_{\mathrm{sing}}$.
\end{itemize}
\end{theorem}

\begin{proof}
The proof is a slight modification of the proof of Proposition 3.8 in \cite{Liu}.
We first prove the theorem in the special case where $f$ is a finite Galois covering.
Let $G$ be the Galois group of $f$.
Then, replacing $k$ by a finite separable extension if necessary, $X$ has a semi-stable model $\mathscr{X}''$ that dominates $\mathscr{X}$ and admits an extension of the $G$-action (see Corollary 2.5 in \cite{Liu}).
We want to modify this to a semi-stable model with {\em inversion-free} action, as follows. 
Suppose an element $\sigma\in G$ of order two interchanges two components $C_1$ and $C_2$ (possibly $C_1=C_2$) intersecting at a node $u$.
Then we blow-up $\mathscr{X}''$ at the closed point $u$; we do this at all such nodes.
The exceptional curves have multiplicity two.
Then we replace $k$ by a ramified quadratic extension $k'$, and take the normalization to obtain a model $\mathscr{X}'$ of $X_{k'}$; it is clear that the $G$-action extends to $\mathscr{X}'$.
The quotient $\mathscr{Y}'=\mathscr{X}'/G$ is a semi-stable model of $Y_{k'}$ (see Proposition 1.6 in \cite{LiuLorenzini}), and the quotient map $\varphi\colon\mathscr{X}'\rightarrow\mathscr{Y}'$ has the desired properties; we postpone the verification of property (c). 

Next, we treat the case where $f$ is separable.
Let $\til{X}$ denote the Galois closure of $f\colon X\rightarrow Y$.
Then, replacing $k$ by a finite separable extension if necessary, we may assume that $\til{X}$ is smooth over $k$.
As in the proof of Proposition 3.8 in \cite{Liu}, replacing $k$ furthermore by a finite separable extension if necessary, one has a semi-stable model $\til{\mathscr{X}}$ of $\til{X}$ that dominates $\mathscr{X}$ and admits an extension of the action of $G=\mathrm{Gal}(\til{X}/Y)$.
As in the first part, we modify $\til{\mathscr{X}}$ to an inversion-free semi-stable model $\til{\mathscr{X}}'$ (after replacing $K$ by a finite separable extension).
Then the obvious map $$\varphi\colon\mathscr{X}'=\til{\mathscr{X}}'/H\rightarrow\mathscr{Y}'=\til{\mathscr{X}}'/G,$$ where $H=\mathrm{Gal}(\til{X}/X)$, gives the desired model of $f$, as we will see soon below.

In general, we decompose $f$ into a finite separable $X\rightarrow Z$ followed by a purely inseparable Frobenius map $Z\rightarrow Y\cong Z^{(p^r)}$(see Proposition 3.5 in \cite{Liu}).
The first part $X \rightarrow Z$ of the decomposition has an $R'$-model $\mathscr{X}'\rightarrow\mathscr{Z}'$ obtained as above. Setting $\mathscr{Y}'=\mathscr{Z}^{\prime (p^r)}$, we find that the composite map $$\varphi\colon\mathscr{X}'\rightarrow\mathscr{Z}'\rightarrow\mathscr{Y}'$$ gives the answer.

The $R'$-morphism $\varphi\colon\mathscr{X}'\rightarrow\mathscr{Y}'$ thus obtained has properties (a) and (b).
In order to show that (c) holds, it suffices to show that neither of the following two situations occurs:
\begin{itemize}
\item[(i)] there exists a double point $u$ of $\mathscr{X}'_0$ that is mapped to a smooth point of $\mathscr{Y}'_0$;
\item[(ii)] there exists a smooth point $u$ of $\mathscr{X}'_0$ that is mapped to a double point of $\mathscr{Y}'_0$.
\end{itemize}
One can see from the construction (due to the `inversion-free' nature) above that the situation (i) does not occur. Finally, situation (ii) is also excluded due to Proposition 1.6 in \cite{LiuLorenzini}.
\end{proof}

\section{Graphs and their stable gonality} \label{graphs}
\begin{se} \label{notgr} Let $G$ be a connected finite graph. In this paper, a graph can have multiple edges (this is sometimes called a ``multigraph'', but we will not use this terminology). We denote the sets of vertices and edges by $\V=\V(G)$ and $\E=\E(G)$, respectively.  We denote by $|G|$ the cardinality $|\V(G)|$ of the vertex set. By $\E(x,y)$ we denote the set of edges connecting two vertices $x,y\in\V(G)$, and more generally, for two subsets $A, B \subseteq \V$, we denote by $\E(A,B)$ the set of edges in $G$ that connect elements from $A$ to elements from $B$: 
$$ \E(A,B) = \bigcup_{x \in A \wedge y \in B}  \E(x,y).  $$ Our graphs are, unless clearly indicated, undirected, i.e., $\E(x,y)=\E(y,x)$. In case we have an oriented edge we will write $(x,y)$ for an edge with source $x$ and target $y$. The set of edges incident to a given vertex $x$ is denoted by $\E_x$. The number of edges in $\E_x$, where, as usual, edges in $\E(x,x)$ (viz., loops) are counted with multiplicity two, is called the \emph{degree} or \emph{valancy} of $x$, and is written $d_x$. A graph is called $k$-\emph{regular} if $d_x=k$ for all $x\in\V$. A graph is called \emph{loopless} if $|\E(x,x)|=0$ for all $x\in V$. Two vertices $x,y$ are called \emph{adjacent} if $|E(x,y)|\geq 1$, and we denote it by $x\sim y$. For a subset $S\subset V$ the \emph{volume} is defined to be $$\vol(S)=\sum_{v\in S}d_v.$$ 
In particular,  $\vol(G)=2 \cdot |\E|$. 

Another important invariant of a graphs is the \emph{genus}, by which we mean the first Betti number $g(G)=|\E|-|\V|+1$. Note that this differs from another convention in graph theory in which ``genus'' means the minimal genus of a Riemann surface in which the graph can be embedded without self-intersection. A graph of genus 0 is called a \emph{tree}.

Functions $f:\V\to\R$, are simply called ``functions on $G$''. These form a finite dimensional vector space, equipped with the standard inner product 
\[\langle f, g\rangle=\sum_{v\in\V(G)}f(v)g(v).\]
\end{se} 

\begin{se} Denote by $A=A_G$ the adjacency matrix of a connected graph $G$ (of which the $(x,y)$-entry is $|E(x,y)|$) and by $D=D_G$ the diagonal matrix with the degrees of the vertices on the diagonal. Then the Laplace operator is defined by $L=L_G=D-A$. 

For any graph, $L_G$ is a real symmetric positive-semidefinite matrix, and therefore has non-negative real eigenvalues. The function $\mathbb{1}$, defined as being identically equal to $1$ on $\V$, is an eigenfunction of $L_G$ with eigenvalue $0$. The other eigenvalues are positive. We order the eigenvalues
\[0=\lambda_0 < \lambda_1\leq\lambda_2...\leq\lambda_{n-1},\]
where $n$ is the number of vertices of the graph. It is the first non-zero eigenvalue which is important for us; we denote it by $\lambda_G:=\lambda_1$.  
\end{se} 

\begin{se} \label{NL}
Sometimes, one uses the \emph{normalized Laplacian} of $G$, defined as \begin{equation*}  L^\sim_G=D_G^{-1/2} L_G D_G^{-1/2}\end{equation*} weighted by vertex degrees (compare Chung \cite{Chung}). 
\end{se}

\begin{definition}
A graph $G$ is called \emph{stable} if all vertices have degree at least $3$. A graph $G'$ is called a \emph{refinement} of $G$ if it can be obtained from $G$ by performing subsequently finitely many times one of the two following operations:
\begin{enumerate}
 \item subdivision of an edge,
 \item addition of a \emph{leaf}, i.e., the addition of an extra vertex and an edge between this vertex and a vertex of the already existing graph.
\end{enumerate}
\end{definition}

\begin{remark}
One of the main tools in this paper is the notion of harmonic morphisms of graphs as developed by Urakawa \cite{Urakawa} and  Baker and Norine \cite{BakerNorine}, and later generalized to harmonic indexed morphisms by Caporaso \cite{Caporaso}.  We will use a terminology that is compatible with that of \cite{Amini}, and different from \cite{Caporaso}, and we will only consider ``unweighted'' graphs in the sense of \cite{Caporaso}. In the appendix, we will discuss the relations between different notions of gonality for graphs. 
\end{remark}

\begin{definition} \label{defmor}
Let $G,G'$ be two loopless graphs.
\begin{enumerate}
 \item
A \emph{finite morphism between $G$ and $G'$} (denoted by $\f:G\to G'$) is a map $$\varphi \colon \V(G)\cup\E(G)\to\V(G')\cup\E(G')$$ such that  $\varphi(\V(G))\subset \V(G')$ and for every $e\in\E(x,y)$, $\varphi(e)\in\E(\varphi(x),\varphi(y))$, together with, for every $e\in\E(G)$, a positive integer $r_\f(e)$, the \emph{index} of $\f$ at $e$.  
\item A finite morphism is called \emph{harmonic} if for every $v\in\V(G)$ there exists a well-defined number, $m_\f(v)$, such that for every $e'\in\E_{\f(v)}(G')$ we have 
\[m_{\f}(v)=\sum_{e\in\E_v,\f(e)=e'}r_\f(e).\]
This does not make sense if $\E_{\f(v)}(G') = \emptyset$, but then we postulate that  $m_{\f}(v)$ can be chosen to be any positive integer. 
\item For a finite harmonic morphism the following number, which is called the \emph{degree} of $\f$, is independent of $v'\in\V(G')$ or $e\in\E(G')$ (\cite{BakerNorine}, Lemma 2.3):
\[\deg\f=\sum_{v\in\f\inv(v')}m_\f(v)=\sum_{e\in\f\inv(e')}r_\f(e).\]
\end{enumerate}
\end{definition}

From the perspective of this paper, it is natural to define the following notion of gonality (this is different from existing notions of gonality, but we will discuss these in the appendix). 

\begin{definition}
A graph $G$ is called \emph{stably} $d$-\emph{gonal} if it has a refinement that allows a degree $d$ finite harmonic morphism to a tree. The \emph{stable gonality} of a graph $G$ is defined to be 
$$ \mathrm{sgon}(G)=  \min\{\deg\f\ | \f \colon G' \to T \} $$ with $G'$  a refinement of $G$ and $\f $ a finite harmonic morphism to a tree $T$.
\end{definition}

\begin{remark}
Although here, finite harmonic morphisms are defined only for loopless graphs, stable gonality is defined for all graphs, as loops can be ``refined away'' by subdividing the loop edges. Alternatively, one may extend the definition of ``harmonic'' to graphs with loops, as in \cite{Amini}. 
\end{remark}

\begin{example} \label{examplegon1}
The ``banana graph'' $B_n$ (see Figure \ref{figbanana}) given by two vertices joined by $n>1$ distinct edges is the intersection dual graph of two rational curves intersecting in $n$ points. The minimal degree of a finite harmonic morphism from $B_n$ to a tree is $n$. However, if we subdivide each edge once, the resulting graph admits such a finite harmonic morphism of degree $2$ to a tree, which is a vertex with $n$ edges sticking out (by identifying the two original vertices). Hence the banana graph has stable gonality equal to $2$. This is compatible with the fact that the banana graph can be the intersection dual graph of both hyperelliptic and non-hyperelliptic (if $n>3$) curves, and these are not distinguished by all subdivisions of their reduction graph.  

This example occurs in nature as the stable reduction of the modular curve $X_0(p)$ over $\Q_p$,  where $n$ is then the number of supersingular elliptic curves modulo $p$. One should observe (\cite{Baker}, 3.6) that stable reduction graphs are naturally \emph{metric} graphs, and as such, the stable reduction graph of $X_0(p)$ is only equal to the (unit-length metrized) banana graph for $p=1\, \mathrm{mod}\, 12$ with $n=(p-1)/12$.   
\begin{center}
\begin{figure}[h]
\includegraphics[width=5cm]{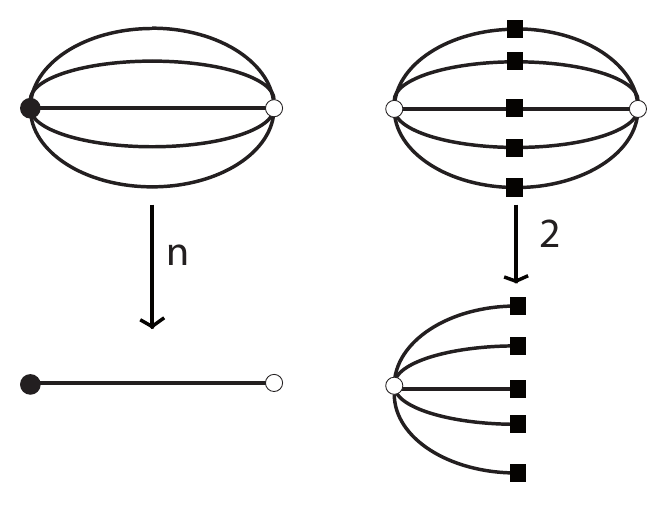}
\caption{A banana graph $B_n$ with a finite harmonic morphism of (minimal) degree $n$, and its subdivision, with a finite harmonic morphism of degree $2$ (all indices are $1$).}
\label{figbanana}
\end{figure}
\end{center}
\end{example}

\begin{example} \label{examplegon2} 
The minimal degree of a finite harmonic morphism from the complete graph $K_4$ to a tree is $4$ (this can be checked by a somewhat tedious enumeration), but by adding leaves, such a morphism of degree $3$ can be constructed, see Figure \ref{figk4}. 
\begin{center}
\begin{figure}[h]
\includegraphics[width=9cm]{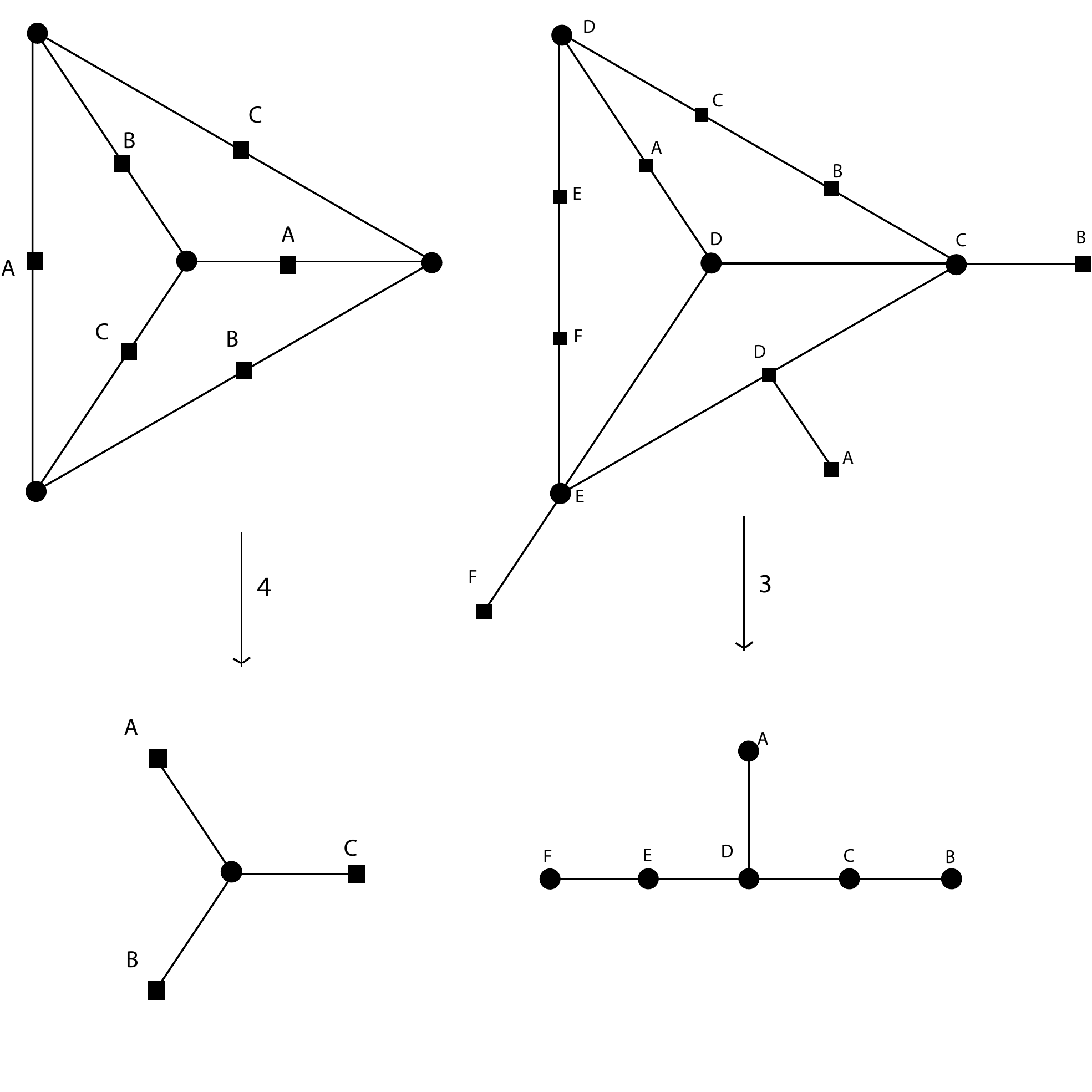}
\caption{A subdivision of $K_4$ with a finite harmonic morphism of degree $4$, and a refinement (with leaves) of $K_4$ with a finite harmonic morphism of degree $3$ (all indices are $1$).}
\label{figk4}
\end{figure}
\end{center}
\end{example}

\section{Comparing curve gonality and graph gonality: proof of Theorem \ref{compare}}
\begin{se} 
Let $\mathscr{X}$ and $\mathscr{Y}$ be $R$-models of geometrically connected projective smooth curves over $k$, and $\varphi\colon\mathscr{X}\rightarrow\mathscr{Y}$ an $R$-morphism.
Let us say $\varphi$ is {\em inversion-free semi-stable} if the following conditions are satisfied:
\begin{itemize}
\item[{\rm (a)}] $\mathscr{X}$ and $\mathscr{Y}$ are semi-stable;
\item[{\rm (b)}] $\varphi$ is finite and surjective; 
\item[{\rm (c)}] $\varphi^{-1}_0((\mathscr{Y}_0)_{\mathrm{sing}})=(\mathscr{X}_0)_{\mathrm{sing}}$.
\end{itemize}
Theorem \ref{theorem-21} says that any finite cover $f\colon X\rightarrow Y$ between geometrically connected projective smooth curves over $k$ admits, after replacing $k$ by a finite separable extension, an inversion-free semi-stable model $f$; moreover, given an arbitrary $R$-model $\mathscr{X}$ of $X$, we can take such an $R$-model $\varphi\colon\mathscr{X}'\rightarrow\mathscr{Y}'$ of $f$ such that $\mathscr{X}'$ dominates $\mathscr{X}$.
\end{se} 

\begin{se} Let $\varphi\colon\mathscr{X}\rightarrow\mathscr{Y}$ be an inversion-free semi-stable model of $f$.
Consider the dual graphs $\Delta:=\Delta(\mathscr{X}_0)$ and $\Gamma=\Delta(\mathscr{Y}_0)$ of the special fibers of $\mathscr{X}_0$ and $\mathscr{Y}_0$, respectively.
The vertices of $\Delta$ (respectively $\Gamma$) correspond to irreducible components of $\mathscr{X}_0$ (respectively $\mathscr{Y}_0$), and two of them are connected by an edge if and only if they intersect.

The morphism $\varphi$ induces the following two set-theoretic maps:
\begin{itemize}
\item since $\varphi$ is finite, it maps each component of $\mathscr{X}_0$ surjectively onto a component of $\mathscr{Y}_0$; in particular, it induces a map $\V(\Delta)\rightarrow \V(\Gamma)$ between the sets of vertices of the graphs;
\item due to the condition (c) above, each double point of $\mathscr{X}_0$ is mapped to a double point of $\mathscr{Y}_0$; that is, we have the map $\E(\Delta)\rightarrow \E(\Gamma)$ between the sets of edges.
\end{itemize}
Thus we obtain a graph map 
$
\phi\colon\Delta\longrightarrow\Gamma.
$
\end{se} 

\begin{se} We now assume that $f$ is \emph{separable}, and define  
the index $r_{\phi}$ for such  $f$.
Let $e\in \E(\Delta)$ be an edge with extremities $v,v'\in \V(\Delta)$.
Let $C,C'$ (respectively $D,D'$) be the components of $\mathscr{X}_0$ (respectively $\mathscr{Y}_0$) corresponding to $v,v'$ (respectively $\phi(v),\phi(v')$), respectively.
The maps $C\rightarrow D$ and $C'\rightarrow D'$ ramify at the intersection point $u$ with the same decomposition group; then define $r_{\phi}(e)$ to be the order of this group.
In this way, $\phi$ becomes a finite morphism of graphs in the sense of Definition \ref{defmor}. 
\end{se} 

\begin{proposition}\label{prop-31}
For a separable $f\colon X\rightarrow Y$ that admits an inversion-free semi-stable model $\varphi\colon\mathscr{X}\rightarrow\mathscr{Y}$, the finite graph morphism $\phi\colon\Delta\rightarrow\Gamma$ constructed above is harmonic of degree $\deg(f)$ $($in the sense of Definition \ref{defmor}$)$.
\end{proposition}

\begin{proof}
Let $v\in \V(\Delta)$ be a vertex, and $C$ (respectively $D$) the component of $\mathscr{X}_0$ (respectively $\mathscr{Y}_0$) corresponding to $v$ (respectively $\phi(v)$).
Let $m_{\phi}(v)$ be the degree of the covering map $C\rightarrow D$.
Then for any edge $e'\in \E(\Gamma)$ emanating from $\phi(v)$, we have
$$
m_{\phi}(v)=\sum_{\phi(e)=e'}r_{\phi}(e), 
$$
and hence $\phi$ is harmonic of degree $\deg(f)$. 
\end{proof}

\begin{corollary}[$=$Theorem \ref{compare}] \label{cor-31}
Let $X$ be a geometrically connected projective smooth curve over $k$, and $\mathscr{X}$ the stable $R$-model of $X$, and let $\Delta(\mathscr{X}_0)$ denote the intersection dual graph of the special fiber. 
Let $\ovl{k}$ be an algebraic closure of $k$.
Then we have
$$
\mathrm{gon}_{\bar{k}}(X)\geq\mathrm{sgon}(\Delta(\mathscr{X}_0)).
$$
\end{corollary}

\begin{proof} Gonality is the minimal degree of a map $f \colon X \rightarrow \PP^1$. Since we work over an algebraically closed field, we can decompose such a map into a separable part $f \colon X \rightarrow Z$ and a purely inseparable part $Z \to Z^{(p^r)} \cong \PP^1$. Since the genus is preserved by the purely inseparable part, we find that $Z \cong \PP^1$, too, and hence the separable part of a general map is a map of lower degree to $\PP^1$. Hence we can restrict to bounding the degree of a separable $f$. 

The assertion now follows from Proposition \ref{prop-31} and the following auxiliary observations.

(1) By Theorem \ref{theorem-21}, for any given finite cover $f\colon X\rightarrow\PP^1_k$, replacing $k$ by a finite separable extension, one has an inversion-free semi-stable model $\varphi\colon\mathscr{X}'\rightarrow\mathscr{P}'$ of $f$ such that $\mathscr{X}'$ dominates $\mathscr{X}$.
In particular, $\Delta(\mathscr{X}'_0)$ gives a graph that arises from $\Delta(\mathscr{X}_0)$ by subdividing some edges (corresponding to blowing up nodes) and/or adding some leaves (corresponding to blowing up smooth points) --- this is exactly the notion of refinement as we have defined it. 

(2) By replacing the base field $k$ by an arbitrary finite extension $k'$, the base-change $\mathscr{X}_{R'}$, where $R'$ is the integral closure of $R$ in $k'$, is a semi-stable model of $X_{k'}$ (see Section 1.5 in \cite{LiuLorenzini}), which obviously gives the same dual graph as $\Delta(\mathscr{X}_0)$.
\end{proof}

\section{The Brill--Noether bound for stable gonality of graphs: proof of Theorem \ref{BNN}} 

One can use this comparison theorem to prove the analog for stable gonality of graphs of the upper bound for the gonality of curves given by Brill--Noether theory: a curve of genus $g$ over an algebraically closed field has gonality bounded above by $\lfloor (g+3)/2 \rfloor$; this was proven in general by Kleiman and Laksov \cite{Kleiman}. To prove this for graphs, we first show that finite harmonic morphisms can be ``refined'', in a sense to be made precise. 

\begin{definition}
For any two refinements $G_1$ and $G_2$ of a graph $G$ let $G_1\vee G_2$ be the set of all common refinements of $G_1$ and $G_2$.
\end{definition}

\begin{definition} A refinement $G'$ of a graph $G$ induces refinements of all of its subgraphs. If $e \in \E(v,w)$ is an edge in $G$ that connect two vertices $v, w \in \V(G)$, denote by $[e]$ the subgraph of $G$ consisting of the vertices $v$ and $w$ joined by the edge $e$. Similarly, for a vertex $v$, we denote by $[v]$ the subgraph which consists of $v$ only. Denote with $G'[x]$ the refinement of $[x]$ in $G'$, and for an edge $e\in\E(v,w)$ denote with $RG'[e]$ the restricted refinement: $$RG'[e]=G'[e]-(G'[v]-[v])-(G'[w]-[w]).$$
\end{definition} 

\begin{definition} \label{defref} A \emph{refinement of a finite harmonic morphism} $\varphi:G \to T$ is a  finite harmonic morphism \[\varphi':G'\to T'\]
such that $G'$ (respectively $T'$) is a refinement of $G$ (respectively $T$), and such that 
\begin{enumerate}
 \item for all $v\in\V(G)$, $\varphi'(v)=\varphi(v)$;
 \item for any $v, w \in \V(G)$ and any edge $e \in \E(v,w)$, every refinement of $[e]$ in $G'$ is mapped to the refinement of $[\varphi(e)]$ in $T'$, viz., 
 $$ \varphi(G'[e]) = T'[\varphi(e)];$$ 
\item\label{indexedge} for every $e\in\E(G)$ and for all $e'\in RG'[e]$, the index  $r_{\varphi'}(e')=r_{\varphi}(e)$.
\end{enumerate}
It follows that $\deg \varphi=\deg\varphi'$.
\end{definition}

\begin{lemma}
Let $\varphi: G \to T$ be a finite harmonic morphism. Then 
\begin{enumerate} 
\item[\textup{(i)}] \label{boompje} for any refinement $T'$ of $T$, there exists a refinement $\varphi' \colon G' \rightarrow T'$ of $\varphi$; 
\item[\textup{(ii)}] \label{lemref} for any refinement $H$ of $G$, there exists a refinement $\varphi': G'\to T'$ of $\varphi$ such that $G'$ is a refinement of $H$. 
\end{enumerate}
\end{lemma}

\begin{proof} For part (i), use the following recipe:
\begin{enumerate} 
\item \label{verfijndeboom} replace every edge $e$ in $G$ by $T'[\varphi (e)]$;
\item put indices such that conditions (\ref{indexedge}) in Definition \ref{defref} is satisfied;
\item extend $\varphi$ in the obvious way to a finite harmonic morphism.  
\end{enumerate}
For part (ii), first choose for every edge $e_0$ in $T$ an element in $$\bigvee_{e\in\varphi^{-1}(e_0)}G'([e]),$$ and replace $e_0$ by this common refinement. Call the resulting new graph $T'$, and then apply part (i). 
\end{proof}

\begin{example} \label{examplerefine} In Figure \ref{figrefine}, one sees a finite harmonic morphism $G \to T$ on the left, where all edges have index $1$, except the indicated edge that has index two. The middle picture is a refinement $H$ of the original graph $G$, and the right hand picture shows the refinement $G' \to T'$ as constructed in Lemma \ref{lemref}. Both morphisms have degree $3$. 
\begin{center}
\begin{figure}[h]
\includegraphics[width=13cm]{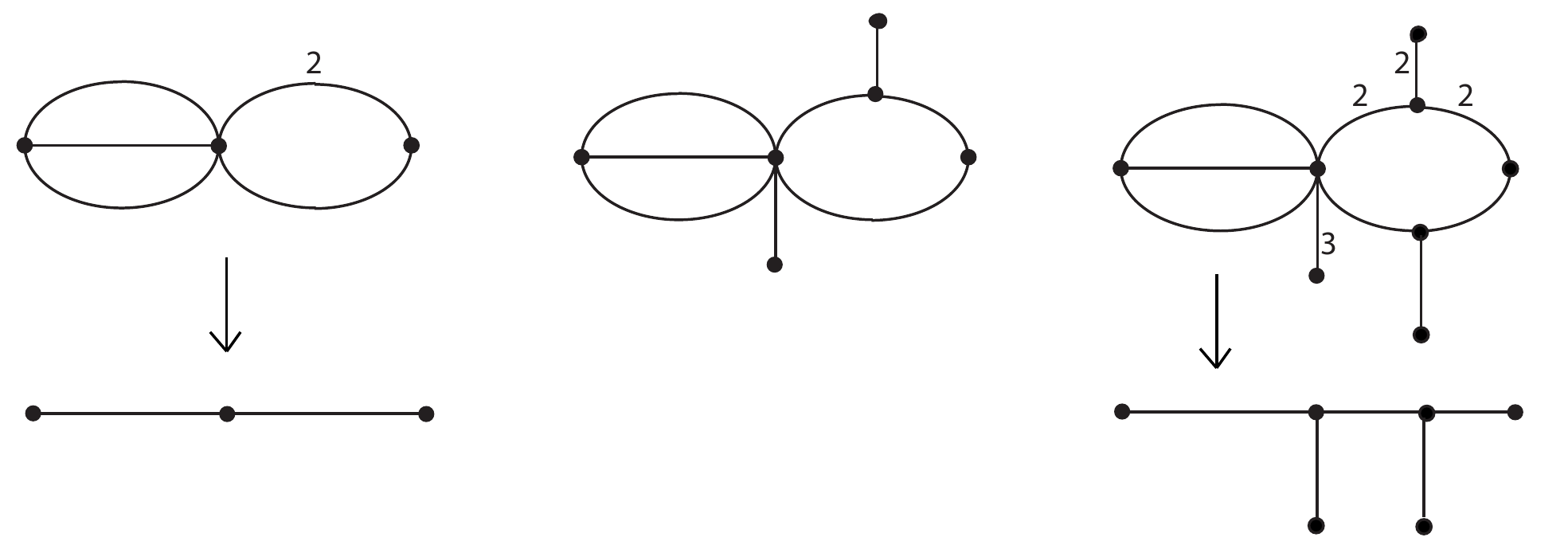}
\caption{Left: a morphism $G \to T$; middle: a refinement $H$ of $G$; right: a refinement $G' \to T'$ with $G'$ a refinement of $H$ (only indices $>1$ are depicted).}
\label{figrefine}
\end{figure}
\end{center}
\end{example}

\begin{corollary} \label{steq} Call two graphs \emph{equivalent} if they are refinements of the same stable graph. This defines an equivalence relation on the set of all graphs of genus at least $2$. The map $\mathrm{sgon}$ is defined on equivalence classes of graphs. 
\end{corollary}

\begin{proof}
Let $G'$ be a refinement of $G$. It follows from the definition that $\mathrm{sgon}(G')\geq\mathrm{sgon}(G)$. Since refinement of morphisms preserves degree, the previous lemmas imply that the other inequality $\mathrm{sgon}(G')\leq\mathrm{sgon}(G)$ also holds. 
\end{proof}

\begin{theorem}[= Theorem \ref{BNN}] \label{BNN2}
For any graph $G$ of genus $g \geq 2$, the Brill--Noether bound holds: $$\mathrm{sgon}(G)\leq\lfloor\frac{g+3}{2}\rfloor.$$
\end{theorem}

\begin{proof}
Since $\mathrm{sgon}$ is defined on the equivalence classes of graphs it is sufficient to prove the bound for one representative of each equivalence class. It is sufficient to show that any stable graph $G$ of genus $g \geq 2$ admits a refinement $G'$ such that there exists a curve $X$ such that $G'$ is the dual graph of the minimal model of $X$. Indeed, since the genus of $X$ equals the genus of $G'$, which equals $g$ (since the genus of a graph doesn't change under refinement), the classical bound $\mathrm{gon}_{\bar{k}}(X) \leq \lfloor(g+2)/3\rfloor$ holds (cf.\ Kleiman--Laksov \cite{Kleiman}). The result follows from $ \mathrm{sgon}(G') \leq \mathrm{gon}_{\bar{k}}(X)$ (Theorem \ref{compare}).

We now show the existence of such a refinement. Let $G$ be a stable graph of genus $g \geq 2$ and let $\Delta_G=\max\{d_x|x\in\V(G)\}$. Choose $g$ edges $e_1,...,e_g$ of $G$ such that $G-\{e_1,...,e_g\}$ is a tree. Replace each edge $e_i$ (connecting two vertices $x_i$ and $y_i$) by two edges $[x_i,v_i]$ and $[w_i,y_i]$, where $v_i$ and $w_i$ are new vertices not connected to any other vertex.  In this way, $G$ is replaced by a tree $T_G$. Choose an embedding of $T_G$ in the Bruhat-Tits  tree $\T$ for  $k=\F_q((t))$, where $q$ satisfies $q+1\geq\Delta_G$. Denote the images of $v_i$ and $w_i$ in $\T$ by the same letters. Now choose hyperbolic elements $\gamma_1,...,\gamma_g$ in $\PGL(2,k)$ such that each $\gamma_i$ acts as translation along a geodesic through $v_i$ and $w_i$, and  $\gamma_i(v_i)=w_i$. Then $\Gamma=\langle \gamma_1,...,\gamma_g\rangle$ is a Schottky group. Denote by $\T_\Gamma$ the subtree of $\T$ spanned by the limit set of $\Gamma$. Then 
$G'\simeq\Gamma\backslash\T_\Gamma$,
where $G'$ is the refinement of $G$ given by subdividing each of the edges $e_1,...,e_g$ once. Also, $G'$ is the intersection dual graph of the minimal model of the Mumford curve corresponding to $\Gamma$ (\cite{Mumford}, page 164).
\end{proof}

It would be interesting to have a purely graph theoretical proof of the above result.  

\begin{remark}
More general lifting results, such as Lemma 6.3 in \cite{Saidi} or Corollary B.3 from \cite{Baker} with Theorem \ref{theorem-21}, also imply the existence of the refinement and the curve. 
\end{remark}

\section{A spectral lower bound for the stable gonality of a graph: proof of Theorem \ref{spectral}}\label{bound}

In this section, we prove an analogue of the Li--Yau bound, viz., a spectral lower bound for the stable gonality of a graph. The basic philosophy of the proof is to find a lower bound for the first Laplace eigenvalue using its variational characterization, in terms of the degree of a finite harmonic morphism $\varphi \colon G' \rightarrow T$ and the minimal ``size'' of the inverse image of the two parts in which the tree gets cut by removing one of its edges. Then a dichotomy occurs: either the minimal such size is large, or there is a vertex with a large inverse image. Initially, ``large'' depends on the maximal vertex degree of the tree $T$, but if this degree is too large, we change the refinement and morphism to produce a lower bound that only depends on the original graph, not the morphism or tree itself. 

We start by studying such ``sizes'' on trees abstractly: 

\begin{definition}
A \emph{measured tree} $(T,\nu)$ is a connected tree $T$ with a probability measure $\nu$ on $\V(T)$. For an edge $e \in \E(T)$, we decompose $T-e$ into its two connected components $T_1(e)$ and $T_2(e)$: 
$$ T-e=T_1(e) \bigsqcup T_2(e). $$
 We define the \emph{size} of an edge $e \in \E(T)$ (w.r.t.\ $\nu$) by 
$$ \size_\nu(e):= \min \{ \nu(T_1(e)), \nu(T_2(e)) \}. $$

Let $c>0$. Call a measured tree $(T,\nu)$ \emph{$c$-thick} if for every vertex $x \in \V(T)$, the graph $T-x$ has a connected component of measure at least $c$.
\end{definition}

\begin{lemma} \label{Amini} A $c$-thick measured tree $(T,\nu)$ has an edge of size at least $c$.  
\end{lemma}

\begin{proof} For any vertex $x \in V(T)$, choose a connected component $C_x$ of $T-x$ of measure at least $c$. Orient the unique edge that connects $x$ to a vertex in $C_x$ in the direction of $x$. By doing this for each vertex, $|T|$ different orientations are assigned to the $|T|-1$ edges of $T$. Hence at least one edge of $T$ is oriented in both directions, and such an edge has size at least $c$. 
\end{proof}

\begin{remark} \label{oldcase}
If $(T,\nu)$ is not $c$-thick, then there exists a vertex $x \in \V(T)$ with $\nu(x) > 1-cd_x$. Indeed, since $T-x$ has $d_x$ components, all of measure less than $c$, we find that $1-\nu(x)=\nu(T-x) < c d_x.$
\end{remark}

The measure we will use counts vertices of $G'$ that belong to the original graph $G$: 

\begin{definition} Let $G$ denote a graph, and $G'$ a refinement of $G$. The probability measure $\mu_G$ on $\V(G')$ is defined by $$\mu_G(A):= \frac{|A \cap \V(G)|}{|G|} \mbox{ for } A \subseteq \V(G').$$
\end{definition} 

\begin{lemma} \label{trivialcase} If $G'$ is a refinement of a graph $G$, and $\varphi \colon G' \rightarrow T$ a finite harmonic morphism to a tree, then $(T,\varphi_* \mu_G)$ is a measured tree, and for any vertex $x \in \V(T)$ , we have $$\deg \varphi \geq \varphi_*\mu_G(x) \cdot |G|. $$
\end{lemma}

\begin{proof}
It suffices to remark that $$\varphi_*\mu_G(x) \cdot |G|= | \varphi^{-1}(x) \cap G |  = \sum_{{v \in G}\atop{\varphi(v)=x}} 1 \leq \sum_{{v \in G}\atop{\varphi(v)=x}} m_\varphi(v) \leq \deg \varphi.$$ 
\end{proof}

The next proposition says that size and degree controls the first eigenvalue of the Laplacian: 

\begin{proposition} \label{Cheeger} If $G'$ is a refinement of a graph $G$, and $\varphi \colon G' \rightarrow T$ a finite harmonic morphism, then for any edge $e \in T$, we have an inequality
$$ \deg \varphi \geq \frac{1}{2}\cdot  \lambda_G \cdot \size_{\varphi_*\mu_G}(e) \cdot |G|.$$ 
\end{proposition}

\begin{proof}

If we let $G_i:=\V(G)\cap \f\inv(\V(T_i(e)))$ for $i=1,2$ then the statement to be proven is equivalent to $$\frac{1}{2}\lambda_G\min(|G_1|,|G_2|)\leq\deg(\f).$$

First, note that the inequality is trivial if $\min(|G_1|,|G_2|)=0$. Now assume the minimum is non-zero. The estimate follows from the variational characterization of $\lambda_G$ via the Rayleigh-quotient,
\[\lambda_G=\inf_{f\perp\mathbb{1}}\frac{\langle f, L f\rangle}{\langle f, f\rangle}=\inf_{f\perp\mathbb{1}}\frac{\sum\limits_{u\sim v}(f(u)-f(v))^2}{\sum_{v}f(v)^2},\]
where notations are as in \ref{notgr}. 
We construct an appropriate function $f$ based on the finite harmonic morphism $\f \colon G' \to T$ and the removed edge $e\in T$, as follows: 
\[f(v)=\left\{\begin{array}{cl}\frac{1}{|G_1|}&\text{if } v\in G_1,\\-\frac{1}{|G_2|}&\text{if } v\in G_2.\end{array}\right.\]
It is easy to check that $f\perp\mathbb{1}$, and therefore,
\begin{align}
\lambda_G &\leq \frac{\sum\limits_{u\sim v}(f(u)-f(v))^2}{\sum\limits_{v}f(v)^2}\nn\\
&= |\E(G_1,G_2)|(\frac{1}{|G_1|}+\frac{1}{|G_2|})\nn\\
&\leq \frac{2|\E(G_1,G_2)|}{\min(|G_1|,|G_2|)}\nn
\end{align}

We finish the proof by showing that $\deg(\f)\geq |\E(G_1,G_2)|$. Suppose an edge $e\in \E(G_1,G_2)$ is replaced in $G'$ by a path, possibly of length $1$. Let us describe this path as a series of edges $e_1,\dots,e_n \in \E(G')$, such that $e_1$ is incident to a vertex in $G_1$, and $e_n$ is incident to a vertex in $G_2$. Then for at least one of the $e_i$ it holds that $\f(e_i)=e$. The desired inequality follows.
\end{proof} 

A lower bound for the degree of $\varphi$ now follows easily if the map gives a $c$-thick measured tree: 

\begin{corollary} \label{firstcase} If $G'$ is a refinement of a graph $G$, and $\varphi \colon G' \rightarrow T$ a finite harmonic morphism such that $(T,\varphi_*\mu_G)$ is $c$-thick, then $$ \deg \varphi \geq \frac{c}{2}\cdot  \lambda_G \cdot |G|.$$ 
\end{corollary}

\begin{proof}
Immediate from Lemma \ref{Amini} and Proposition \ref{Cheeger}. 
\end{proof} 

\begin{remark} If the tree is \emph{not} $c$-thick, we know from Remark \ref{oldcase} that the tree has a vertex $x \in T$ with ``large'' measure: $\varphi_* \mu_G(x) \geq 1-cd_x.$ Putting, for example, $c=1/(\Delta_T+1)$ (where $\Delta_T$ is the maximal vertex degree in $T$), the previous results gives a non-trivial lower bound on the degree of $\varphi$ in terms of $\lambda_G, |G|$ and $\Delta_T$ of the form
$$\deg(\varphi) \geq \min\left\{ \frac{\lambda_G}{2},1\right\} \frac{|G|}{\Delta_T+1}.$$ However, we want to find a bound that solely depends on $G$, not on $T$. For this, in the next proposition, we engineer another harmonic morphism from a different refinement of $G$, but whose degree is controlled by that of $\varphi$. 
\end{remark}

\begin{proposition} \label{change} Let $G$ denote a graph with maximal degree $\Delta_G$. Let $A, B,C>0$ be constants such that $A+B+C \leq 1$. If $G'$ is a refinement of $G$, and $\varphi \colon G' \rightarrow T$ a finite harmonic morphism such that 
\begin{enumerate}  
\item[\textup{(i)}] $(T,\varphi_*\mu_G)$ is not $(C/2)$-thick; and  
\item[\textup{(ii)}] all vertices $x \in T$ have measure $\varphi_*\mu_G(x)<B,$ 
\end{enumerate} then there exists a refinement  $G^\#$ of $G$, a tree $T^\#$, and a finite harmonic morphism $\f^\#:G^\#\to T^\#$ such that 
\begin{enumerate} 
\item[\textup{(a)}] $\deg \f^\#\leq \Delta_G\deg\f$; and 
\item[\textup{(b)}] there exists an edge $e^\#$ of $T^\#$ with $\size_{\varphi_*^\# \mu_G} (e^\#) \geq A/2.$ \end{enumerate}
\end{proposition}

We postpone the proof to the next section, and first discuss the main corollary.

\begin{corollary}[ = Theorem \ref{spectral}] \label{deafschatting}
Let $G$ be a graph with maximal vertex degree $\Delta_G$ and first Laplace eigenvalue $\lambda_G$. The stable gonality of $G$ is bounded from below by
\[\sgon(G)\geq \left\lceil \frac{\lambda_G}{\lambda_G+4(\Delta_G+1)}|G| \right \rceil.\]
\end{corollary}

\begin{proof}
Let $G$ be a graph and let $\varphi:G'\to T$ be a finite harmonic morphism, and let $A,B,C$ be as in Proposition \ref{change}. If $(T,\varphi_*\mu_G)$ is $C/2$-thick, then by Corollary \ref{firstcase}, we find 
$$ \deg \varphi \geq \gamma \cdot |G| \mbox{ with } \gamma:=\frac{C\lambda_G }{4}.$$
On the other hand, if there is a vertex $x \in T$ with $\varphi_*\mu_G(x) \geq B,$ then by Lemma \ref{trivialcase}, we get 
$$\deg \varphi \geq  \beta \cdot |G| \mbox{ with } \beta:=B.$$
In the remaining case, Proposition \ref{change} implies that 
$$ \deg \varphi \geq \frac{1}{\Delta_G} \deg \varphi^\# \geq \frac{\lambda_G}{2\Delta_G}\size_{\varphi^\#_*\mu_G}(e^\#)|G| \geq\alpha \cdot  |G| \mbox{ with } \alpha:=\frac{A\lambda_G }{4 \Delta_G},$$
where the second inequality follows from Proposition \ref{Cheeger}. 

Translating the constraints $A,B,C>0$ and $A+B+C \leq 1$, we conclude that it always holds that 
\begin{align} \label{simp} \frac{\deg\f}{|G|}\geq \max_{{\alpha, \beta, \gamma >0}\atop{a \alpha + b \beta + c \gamma \leq 1}}\min(\alpha, \beta, \gamma). \end{align}
with $a = 4\Delta_G/\lambda_G$, $b=1$ and $c=4/\lambda_G$. 
The maximum in (\ref{simp}) is achieved for $\alpha=\beta=\gamma$ and $a \alpha+ b \beta+ c \gamma = 1$, and plugging this back into (\ref{simp})  gives the result.
\end{proof}

\section{Proof of Proposition \ref{change}} \label{ppchange}

In the proof of Proposition \ref{change} we will use the following concept several times to construct new graphs from old:

\begin{definition} Let $H_1,\dots ,H_n$ denote $n$ different graphs and let $(v^1_1,\dots ,v^n_1),\dots,(v^1_m,\dots,v^n_m)\in \V(H_1)\times\dots\times\V(H_n)$ denote tuples of their vertices. The graph $H$  obtained by \emph{gluing} $H_i$ along these vertices is defined to be
$$H:=\ \quotient{\bigsqcup\limits_{i=1}^n H_i}{\langle v^1_1=\dots=v^n_1,\dots,v^1_m=\dots=v^n_m\rangle }.$$
\end{definition}

\begin{lemma}\label{restrictieisharmonic}
Let $\f:G\to T$ be a finite harmonic morphism and let $T_0\subset T$ be a connected subgraph. Then the restriction of $\f$ to any of the connected components of $\varphi^{-1}(T_0)$ is a finite harmonic morphism to $T_0$.
\end{lemma}

\begin{proof}
Let $T_0$ be such a subgraph and let $\f_0:G_0\to T_0$ be the restriction of $\f$ to one of the connected components $G_0$ of $\f^{-1}(T_0)$. 
Then $m_{\varphi_0}(v)$ is well-defined for all $v \in V(G_0)$, namely, all edges $e\in\E_v(G)$ which are mapped to an edge $e'\in E_{\varphi(v)}(T_0)$ are contained in $\E_v(G_0)$, so that in particular $m_{\f_0}(v)=m_\f(v)$. 
\end{proof}

\begin{proof}[Proof of Proposition \ref{change}.] The proof consists of various steps, in which the new morphism is constructed from the original map. As the proof proceeds, we will show the constructions on an explicit non-trivial example. The example is printed in smaller font, and the switch between example and main proof is indicated by a diamond ($\Diamond$). 
The basic idea of the proof is this: the hypothesis of the Proposition (which is the remaining bad case from the point of view of getting useful bounds) is that there exists a vertex in the tree (with controlled measure) such that all connected components of its complement have too small measure to give a useful bound. We collect these components (and their preimages) into two sets, which we call the ``left'' and ``right'' parts, such that each of these parts has a larger, more useful, measure. For technical reasons, one should discard leaves first. We want to map each of these parts as a whole onto a ``smaller'' tree, but this will increase the degree of the map. We control it as follows: in the new tree $T^\#$, all of these components are glued to a central vertex. The original morphism $\varphi$ is split into local parts over each of the components, and these are refined to a harmonic morphism over the left or right parts of the new tree. Then the left and right parts are glued together over the new tree, where the indices at preimages of the central point are redefined, and maybe some leaves are added, to make the result into a harmonic morphism. One then checks that the degree of the new morphism hasn't increased too much, but that both edges sticking out of the central vertex have large enough measure to give a useful bound. \\ 

\begin{footnotesize}
\noindent \textbf{Example.}\ We start with a graph $G$ of the following form 
\begin{center} 
\begin{tikzpicture}[scale=0.4,every node/.style={transform shape}]
\SetVertexNormal[Shape      = circle,
                 FillColor  = black,
                 LineWidth  = 1pt]
   \Vertex[x=1 ,y=9,L=b]{1}
   \Vertex[x=5 ,y=7,L=a]{2}
   \Vertex[x=9,y=9,L=d]{3}
   \Vertex[x=7 ,y=5,L=f]{4}
   \Vertex[x=7 ,y=2,L=g]{5}
   \Vertex[x=9 ,y=5,L=e]{6}
   \Vertex[x=5 ,y=4,L=a]{7}
   \Vertex[x=3 ,y=2,L=h]{8}
    \Vertex[x=1 ,y=5,L=b]{9}
    \Vertex[x=3 ,y=5,L=h]{10}
\SetUpEdge[lw = 1pt, color = black]
   \Edge(1)(2)
  \Edge(1)(10)
   \Edge(2)(10)
    \Edge(2)(7)
    \Edge(2)(4)
    \Edge(2)(3)
    \Edge(3)(4)
     \Edge(3)(6)
     \Edge(4)(5)
     \Edge(5)(6)
     \Edge(5)(7)
     \Edge(6)(7)
     \Edge(7)(8)
     \Edge(7)(9)      
     \Edge(8)(9)
     \Edge(8)(10)
 \end{tikzpicture}            
\end{center}
The original graph has $|G|=10$ vertices and maximal vertex degree $\Delta_G=5$. 

We consider the following refinement $G'$ of the graph $G$, and the harmonic morphism $\varphi$ to the tree $T$ as indicated in the following picture. 

We use the following \emph{display conventions:} the original vertices and edges are bold, contrary to subdivision vertices and vertices and edges from leaves. The label on a vertex indicates to what vertex in the tree it is mapped. A gray square box with a number on an edge indicates that this edge has index equal to that number; if there is no box, then the index is $1$. 

\begin{center}
\begin{tikzpicture}[scale=0.7,every node/.style={transform shape}]
\SetVertexNormal[Shape      = circle, LineWidth  = 2pt]
   \Vertex[x=1 ,y=9,L=b]{1}
   \Vertex[x=5 ,y=7,L=a]{2}
   \Vertex[x=9,y=9,L=d]{3}
   \Vertex[x=7 ,y=5,L=f]{4}
   \Vertex[x=7 ,y=2,L=g]{5}
   \Vertex[x=9 ,y=5,L=e]{6}
   \Vertex[x=5 ,y=4,L=a]{7}
   \Vertex[x=3 ,y=2,L=h]{8}
    \Vertex[x=1 ,y=5,L=b]{9}
     \Vertex[x=3 ,y=5,L=h]{10}
     
      \SetVertexNormal[Shape      =circle, LineWidth  = 1pt]

          \renewcommand{\VertexInterMinSize}{10pt}  
          
   \Vertex[x=0,y=10,L=B]{11}      
      \Vertex[x=0,y=6,L=B]{12}              
   \Vertex[x=4,y=9,L=e]{13}      
   \Vertex[x=6,y=9,L=g]{14}      
   \Vertex[x=4,y=2,L=d]{15}  
      \Vertex[x=6,y=2,L=f]{16}  
      
      \Vertex[x=3,y=3.5,L=H]{17}
      \Vertex[x=5,y=5.5,L=c]{18}
            \Vertex[x=2,y=7,L=a]{19}
                  \Vertex[x=2,y=3.5,L=a]{20}
                        \Vertex[x=8,y=7,L=a]{21}
                              \Vertex[x=9,y=7,L=a]{22}
                                    \Vertex[x=8,y=3.5,L=a]{23}
                                          \Vertex[x=7,y=3.5,L=a]{24}

  \Vertex[x=1,y=8,L={\tiny c}]{19a}
    \Vertex[x=1,y=7.5,L={\tiny d}]{19b}
      \Vertex[x=1,y=7,L={\tiny e}]{19c}
        \Vertex[x=1,y=6.5,L={\tiny f}]{19d}
             \Vertex[x=1,y=6,L={\tiny g}]{19e}
\Vertex[x=1,y=4,L={\tiny c}]{20a}
    \Vertex[x=1,y=3.5,L={\tiny d}]{20b}
      \Vertex[x=1,y=3,L={\tiny e}]{20c}
        \Vertex[x=1,y=2.5,L={\tiny f}]{20d}
             \Vertex[x=1,y=2,L={\tiny g}]{20e}

               \Vertex[x=10,y=8,L={\tiny b}]{22a}
      \Vertex[x=11,y=8,L={\tiny B}]{22aa}
    \Vertex[x=10,y=7.5,L={\tiny h}]{22b}
       \Vertex[x=11,y=7.5,L={\tiny H}]{22bb}
      \Vertex[x=10,y=7,L={\tiny c}]{22c}
        \Vertex[x=10,y=6.5,L={\tiny f}]{22d}
             \Vertex[x=10,y=6,L={\tiny g}]{22e}
       
               \Vertex[x=9,y=4,L={\tiny b}]{23a}
      \Vertex[x=10,y=4,L={\tiny B}]{23aa}
    \Vertex[x=9,y=3.5,L={\tiny h}]{23b}
       \Vertex[x=10,y=3.5,L={\tiny H}]{23bb}
      \Vertex[x=9,y=3,L={\tiny c}]{23c}
        \Vertex[x=9,y=2.5,L={\tiny f}]{23d}
             \Vertex[x=9,y=2,L={\tiny d}]{23e}

               \Vertex[x=7.3,y=7.5,L={\tiny b}]{21a}
      \Vertex[x=6.8,y=7.5,L={\tiny B}]{21aa}
    \Vertex[x=7.3,y=7,L={\tiny h}]{21b}
       \Vertex[x=6.8,y=7,L={\tiny H}]{21bb}
      \Vertex[x=7.3,y=6.5,L={\tiny c}]{21c}
        \Vertex[x=8,y=6,L={\tiny f}]{21d}
             \Vertex[x=8.5,y=6,L={\tiny g}]{21e}

               \Vertex[x=6.4,y=4,L={\tiny b}]{24a}
      \Vertex[x=5.9,y=4,L={\tiny B}]{24aa}
    \Vertex[x=6.4,y=3.5,L={\tiny h}]{24b}
       \Vertex[x=5.9,y=3.5,L={\tiny H}]{24bb}
      \Vertex[x=6.4,y=3,L={\tiny c}]{24c}
        \Vertex[x=7.5,y=4.3,L={\tiny f}]{24d}
             \Vertex[x=7.5,y=3.8,L={\tiny g}]{24e}

\draw[line width=1pt,->] (5,2) -- (5,0) node[pos=.5,right] {{\Large $\varphi$}};
             
 
 \SetVertexNormal[Shape      = circle, LineWidth  = 1pt]

 \Vertex[x=3,y=0,L=B]{q1}
 \Vertex[x=4,y=-1,L=b]{q2}
 \Vertex[x=6,y=0,L=c]{q3}
 \Vertex[x=5,y=-2,L=a]{q4}
 \Vertex[x=3,y=-4,L=H]{q5}
 \Vertex[x=4,y=-3,L=h]{q6}
 \Vertex[x=7,y=-1,L=d]{q7}
 \Vertex[x=8,y=-2,L=e]{q8}
 \Vertex[x=7,y=-3,L=f]{q9}
 \Vertex[x=6,y=-4,L=g]{q10}
 
 \SetUpEdge[lw         = 2pt,
           color      = black,
           labelcolor = lightgray]

   \Edge(1)(2)
   \Edge(2)(10)
    \Edge(2)(4)
    \Edge(2)(3)
     \Edge(5)(7)
     \Edge(6)(7)
     \Edge(7)(8)
     \Edge(7)(9)      
            
\SetUpEdge[lw         = 1pt,
           color      = black,
           labelcolor = lightgray]              
  \Edge(1)(11)
    \Edge(9)(12)
    \Edge(2)(13) \Edge(2)(14) 
    \Edge(7)(15) \Edge(7)(16)
    
      \Edge(q1)(q2) \Edge(q2)(q4) \Edge(q5)(q6) \Edge(q6)(q4) \Edge(q4)(q3) \Edge(q4)(q7) \Edge(q4)(q8) \Edge(q4)(q9) \Edge(q4)(q10)
    
    \SetUpEdge[lw         = 2pt,
           color      = black,
          labelcolor = lightgray]

    \Edge(1)(19) \Edge(19)(10)
    \Edge(2)(18) \Edge(18)(7)
    \Edge(3)(21) \Edge(21)(4)
    \Edge(3)(22) \Edge(22)(6)
    \Edge(5)(23) \Edge(23)(6)
    \Edge(8)(20) \Edge(20)(9)
    \Edge(4)(24) \Edge(24)(5)
    \Edge[label={\tiny\textbf{2}}](8)(17) \Edge[label={\tiny\textbf{2}}](17)(10){2}

      \SetUpEdge[lw         = 1pt, color      = black]

    \Edge(19)(19a)   \Edge(19)(19b)   \Edge(19)(19c)   \Edge(19)(19d)   \Edge(19)(19e)  
      \Edge(20)(20a)   \Edge(20)(20b)   \Edge(20)(20c)   \Edge(20)(20d)   \Edge(20)(20e)  
      \Edge(22)(22a)   \Edge(22)(22b)   \Edge(22)(22c)   \Edge(22)(22d)   \Edge(22)(22e)  \Edge(22a)(22aa) \Edge(22b)(22bb)
      \Edge(23)(23a)   \Edge(23)(23b)   \Edge(23)(23c)   \Edge(23)(23d)   \Edge(23)(23e)  \Edge(23a)(23aa) \Edge(23b)(23bb)
      \Edge(21)(21a)   \Edge(21)(21b)   \Edge(21)(21c)   \Edge(21)(21d)   \Edge(21)(21e)  \Edge(21a)(21aa) \Edge(21b)(21bb)
        \Edge(24)(24a)   \Edge(24)(24b)   \Edge(24)(24c)   \Edge(24)(24d)   \Edge(24)(24e)  \Edge(24a)(24aa) \Edge(24b)(24bb)

    \end{tikzpicture}
\end{center}
This map has degree $\deg \varphi = 8$. The push-down $\varphi_* \mu_G$ of the normalized counting measure of $G$ to $T$ takes the following value on the indicated vertices: 
$$ \varphi_*\mu(v) = \left\{ \begin{array}{ll} 0 & \mbox{ if } v\in\{B,c,H\}, \\ 1/10 & \mbox{ if } v \in \{d,e,f,g\}, \\ 1/5 & \mbox{ if } v \in \{a,b,h\}. \end{array} \right. $$
If we set $$ A=1/5, B=3/10, C=5/10, $$
then the set-up satisfies conditions (i) and (ii) of Proposition \ref{change}. Indeed, The push-down measure is not $C/2=1/4$-thick (all connected components of $T-a$ are of measure at most $1/5$), and all vertices have measure smaller than $B=3/10$. 

In the proof, we construct a new refinement $G^\#$ of $G$ and a new tree $T^\#$ with a finite harmonic morphism $\varphi^\# \colon G^\# \rightarrow T^\#$ of degree $\deg \varphi^\#< \Delta_G \deg \varphi =40$ with an edge of size $\geq A/2=1/10$. $\Diamond$
\end{footnotesize} \\

Write $\nu=\varphi_* \mu_G$. Since $(T,\nu)$ is not $C/2$-thick, we can choose a vertex $x_0 \in \V(T)$ such that all components of $T-x_0$ have measure $<C/2$. 
Let $G^s\subset G'$ be the refinement of $G$ which only comprises the subdivided edges of $G$ in $G'$, and let $T^s$ denote its image $\f(G^s)$. Observe that $x_0\in\V(T^s)$, and let $d$ denote the degree of $x_0$ in $T^s$. Also note that the maximal degree of $G^s$ is the same as that of $G$, i.e., $\Delta_{G^s}=\Delta_G$. 
Denote by $T^s_1,\dots, T^s_{d}$ the connected components  of $T^s-x_0$. We divide the connected components into two sets of approximately the same measure. Since $\nu(x_0)<B$ and $\nu(T_i)<C/2$ for all $i=1,\dots,d$, it is possible to find a partition $I_L\cup I_R=\{1,\dots,d\}$, such that 
\[\min\left\{\nu(\bigcup_{i\in I_L}T^s_i),\nu(\bigcup_{i\in I_R}T^s_i)\right\} \geq\frac{1-\nu(x_0)}{2}-\frac{C}{2}>\frac{1-B}{2}-\frac{C}{2}\geq \frac{A}{2}.\]

\begin{footnotesize}
\noindent \textbf{Example (continued).}\ In the example, $G^s$ is the graph 
\begin{center}
\begin{tikzpicture}[scale=0.4,every node/.style={transform shape}]
\SetVertexNormal[Shape      = circle,
                 FillColor  = black,
                 LineWidth  = 1pt]
                 
              \Vertex[x=1 ,y=9,L=b]{1}
   \Vertex[x=5 ,y=7,L=a]{2}
   \Vertex[x=9,y=9,L=d]{3}
   \Vertex[x=7 ,y=5,L=f]{4}
   \Vertex[x=7 ,y=2,L=g]{5}
   \Vertex[x=9 ,y=5,L=e]{6}
   \Vertex[x=5 ,y=4,L=a]{7}
   \Vertex[x=3 ,y=2,L=h]{8}
    \Vertex[x=1 ,y=5,L=b]{9}
     \Vertex[x=3 ,y=5,L=h]{10}

     \SetVertexNormal[Shape      = circle,
                 FillColor  = white,
                 LineWidth  = 1pt]

      \Vertex[x=3,y=3.5,L=$\mbox{ }$]{17}
      \Vertex[x=5,y=5.5,L=$\mbox{ }$]{18}
            \Vertex[x=2,y=7,L=$\mbox{ }$]{19}
                  \Vertex[x=2,y=3.5,L=$\mbox{ }$]{20}
                        \Vertex[x=8,y=7,L=$\mbox{ }$]{21}
                              \Vertex[x=9,y=7,L=$\mbox{ }$]{22}
                                    \Vertex[x=8,y=3.5,L=$\mbox{ }$]{23}
                                          \Vertex[x=7,y=3.5,L=$\mbox{ }$]{24}
 \SetUpEdge[lw         = 1pt, color      = black ]

   \Edge(1)(2)
   \Edge(2)(10)
    \Edge(2)(4)
    \Edge(2)(3)
     \Edge(5)(7)
     \Edge(6)(7)
     \Edge(7)(8)
     \Edge(7)(9)      
          
    \Edge(1)(19) \Edge(19)(10)
    \Edge(2)(18) \Edge(18)(7)
    \Edge(3)(21) \Edge(21)(4)
    \Edge(3)(22) \Edge(22)(6)
    \Edge(5)(23) \Edge(23)(6)
    \Edge(8)(20) \Edge(20)(9)
    \Edge(4)(24) \Edge(24)(5)
    \Edge(8)(17) \Edge(17)(10)
    \end{tikzpicture}
\end{center}
and $T^s$ is the graph 
\begin{center}
\begin{tikzpicture}[scale=0.7,every node/.style={transform shape}]
  
 \SetVertexNormal[Shape      = circle, LineWidth  = 1pt]

 \Vertex[x=4,y=-1,L=b]{q2}
 \Vertex[x=6,y=0,L=c]{q3}
 \Vertex[x=5,y=-2,L=a]{q4}
 \Vertex[x=3,y=-4,L=H]{q5}
 \Vertex[x=4,y=-3,L=h]{q6}
 \Vertex[x=7,y=-1,L=d]{q7}
 \Vertex[x=8,y=-2,L=e]{q8}
 \Vertex[x=7,y=-3,L=f]{q9}
 \Vertex[x=6,y=-4,L=g]{q10}
 
  \SetUpEdge[lw         = 1pt, color      = black ]

\Edge(q2)(q4) \Edge(q5)(q6) \Edge(q6)(q4) \Edge(q4)(q3) \Edge(q4)(q7) \Edge(q4)(q8) \Edge(q4)(q9) \Edge(q4)(q10)

     \end{tikzpicture}
\end{center}
We can choose $x_0$ to be the central vertex labelled ``a'', so that all components of $T-x_0$ have measure $<C/2=1/4$ (indeed, they have measure $0$, $1/5$ or $1/10$), and divide the $d=7$ connected components into ``left'' and ``right'' as follows: 
\begin{align*}  & x_0 = a, \\ & I_L =\left\{T_1=\{b\},T_2=\{h,H\}\right\},\\ & I_R =\left\{T_3=\{c\},T_4=\{d\},T_5=\{e\},T_6=\{f\},T_7=\{g\}\right\}. \end{align*}
Then the total measure of the left part is $2/5$ and the total measure of the right part is $2/5$, which is larger than $A/2=1/10$. 
$\Diamond$
\end{footnotesize}\\

We now show how to construct the different pieces of the new map $\f^\# \colon G^\# \rightarrow T^\#$. \\

\noindent \textbf{The construction of $T^\#$.}\ For each $i=1,\dots,d$ let $y_i\in\V(T^s_i)$ be the unique vertex which is adjacent to $x_0$ in $T^s$. A new graph $S^\#$ is obtained by gluing all $T^s_i$ together at the $y_i$, and adding a leaf at the image of the $y_i$. Call $x$ the new vertex of the added leaf. Take two copies $(S^{\#,1},x_1)$ and $(S^{\#,2},x_2)$ of the pair $(S^\#,x)$, and glue them together at $x_1$ and $x_2$ to obtain a tree $T^\#$. Let $X_0$ be the image of $x_1$ and $x_2$ in $T^\#$, and call the image of $S^{\#,1}$ in $T^\#$  the ``left part''  $T^\#_L$, and the image of $S^{\#,2}$ in $T^\#$ the ``right part'' $T^\#_R$. \\ 

\begin{footnotesize}
\noindent \textbf{Example (continued).}\ In the example, $S^\#$ is the graph given by gluing all components $T_1,\dots,T_7$ of $T^s-x_0$ along their vertex adjacent to $x_0$, and adding a leaf. All but the component $T_2=\{h,H\}$ are isolated vertices, so the result is a segment isomorphic to $T_2$, connected to a new vertex $x$ at $h$: $S^\#$ is 
\begin{tikzpicture}[scale=0.4,every node/.style={transform shape}, baseline=-\the\dimexpr\fontdimen22\textfont2\relax ]
  \SetVertexNormal[Shape      = circle, LineWidth  = 1pt]
 \Vertex[x=5,y=0,L=x]{q4}
 \Vertex[x=3,y=0,L=$\mbox{ }$]{q5}
 \Vertex[x=4,y=0,L=$\mbox{ }$]{q6}
     \SetUpEdge[lw         = 1pt, color      = black, labelcolor = lightgray]      
\Edge(q5)(q6) \Edge(q6)(q4) 
 \end{tikzpicture}
. $T^\#$ is the graph given by gluing two copies of $(S^\#,x)$ along the common vertex $x$, so $T^\#$ is
 \begin{tikzpicture} [scale=0.4,every node/.style={transform shape}, baseline=-\the\dimexpr\fontdimen22\textfont2\relax ]
 \SetVertexNormal[Shape      = circle,LineWidth  = 1pt]
 \Vertex[x=5,y=0,L=$X_0$]{center}
 \Vertex[x=4,y=0,L={\mbox{ }}]{links1}
 \Vertex[x=3,y=0,L={\mbox{ }}]{links2}
 \Vertex[x=6,y=0,L={\mbox{ }}]{rechts1}
  \Vertex[x=7,y=0,L={\mbox{ }}]{rechts2}
 \SetUpEdge[lw         = 1pt, color      = black ]
\Edge(center)(links1) \Edge(links1)(links2) \Edge(center)(rechts1) \Edge(rechts1)(rechts2) 
    \end{tikzpicture}
.$\Diamond$
\end{footnotesize}\\

\noindent \textbf{The construction of $G^\#$.}\ 
For each $i=1,\dots,d$, let $S_i$ be the subtree of $T^s$ obtained by adding to $T^s_i$ the (unique) edge in $\E(x_0,y_i)$. The subgraph $\f^{-1}(S_i)\subset G'$ might be disconnected; let $G_i''$ be the union of the connected components of $\f^{-1}(S_i)$ for which the set of edges has a non-empty intersection with $\E(G^s)$. By Lemma \ref{restrictieisharmonic} the restriction $$\varphi_i'':=\f|_{G_i''}: G_i''\to S_i.$$ is finite harmonic.

Now observe that $S^\#$ is a refinement of $S_i$, so by Lemma \ref{boompje} there exists a finite harmonic refinement morphism (in the sense of Definition \ref{defref}) 
\[\varphi^\#_i: G_i^\#\to S^\#,\]
with an inclusion map $\iota_i:G_i''\to G_i^\#$.  \\

\begin{footnotesize}
\noindent \textbf{Example (continued).}\ For each of the seven connected components $T_i$, we display the construction of the ``local'' components $S_i$, $G_i''$ and the refinement morphism $ \varphi_i^\# \colon G_i^\# \rightarrow S^\#$ in Table \ref{ttable}. The construction for $i=5,6,7$ is entirely similar to the one for $i=4$ (with the label $d$ replaced by $e,f,g$, respectively), so we don't list it in the table. $\Diamond$
\end{footnotesize}

\newpage

{
\begin{table}[h] \caption{The local constructions relating to the maps $ \varphi_i^\# \colon G_i^\# \rightarrow S^\#$ ($i=5,6,7$ are similar to $i=4$)} \label{ttable}

\begin{tabular}{lMMMM}
\hline
$i$ & $1 (\in I_L)$ & $2 (\in I_L)$ & $3 (\in I_R)$ & $4 (\in I_R)$ \\ \hline
$S_i$ &\begin{tikzpicture}[scale=0.6,every node/.style={transform shape}, baseline=-\the\dimexpr\fontdimen22\textfont2\relax ]
  \SetVertexNormal[Shape      = circle, LineWidth  = 1pt]
 \Vertex[x=5,y=-1,L=a]{q4}
 \Vertex[x=4,y=-1,L=b]{q6}
     \SetUpEdge[lw         = 1pt, color      = black, labelcolor = lightgray]    
     \Edge(q6)(q4)   
 \end{tikzpicture} 
 
 &  
 
 \begin{tikzpicture}[scale=0.6,every node/.style={transform shape}, baseline=-\the\dimexpr\fontdimen22\textfont2\relax ]
  \SetVertexNormal[Shape      = circle, LineWidth  = 1pt]
 \Vertex[x=5,y=-1,L=a]{q4}
 \Vertex[x=3,y=-1,L=H]{q5}
 \Vertex[x=4,y=-1,L=h]{q6}
     \SetUpEdge[lw         = 1pt, color      = black, labelcolor = lightgray]      
\Edge(q5)(q6) \Edge(q6)(q4) 

 \end{tikzpicture}
 
 &
 
 \begin{tikzpicture}[scale=0.6,every node/.style={transform shape}, baseline=-\the\dimexpr\fontdimen22\textfont2\relax ]
  \SetVertexNormal[Shape      = circle, LineWidth  = 1pt]
 \Vertex[x=5,y=-1,L=c]{q4}
 \Vertex[x=4,y=-1,L=a]{q6}
     \SetUpEdge[lw         = 1pt, color      = black, labelcolor = lightgray]    
     \Edge(q6)(q4)   
 \end{tikzpicture}
 
  &
  
  \begin{tikzpicture}[scale=0.6,every node/.style={transform shape}, baseline=-\the\dimexpr\fontdimen22\textfont2\relax ]
  \SetVertexNormal[Shape      = circle, LineWidth  = 1pt]
 \Vertex[x=5,y=-1,L=d]{q4}
 \Vertex[x=4,y=-1,L=a]{q6}
     \SetUpEdge[lw         = 1pt, color      = black, labelcolor = lightgray]    
     \Edge(q6)(q4)   
 \end{tikzpicture} 
 
   \\ \hline \\
   
$G_i''$ & 

\begin{tikzpicture}[scale=0.6,every node/.style={transform shape}, baseline=-\the\dimexpr\fontdimen22\textfont2\relax ]
  \SetVertexNormal[Shape      = circle, LineWidth  = 1pt]
 \Vertex[x=1 ,y=7,L=b]{1}
   \Vertex[x=5 ,y=5,L=a]{2}
   \Vertex[x=5 ,y=2,L=a]{7}
    \Vertex[x=1 ,y=3,L=b]{9}     
       \Vertex[x=2,y=5,L=a]{19}
                  \Vertex[x=2,y=1.5,L=a]{20}
     
 \SetUpEdge[lw         = 1pt, color      = black, labelcolor = lightgray]      
     \Edge(1)(19)
     \Edge(1)(2)
     \Edge(9)(7)
     \Edge(9)(20)
        \end{tikzpicture}

& 

 \begin{tikzpicture}[scale=0.6,every node/.style={transform shape}, baseline=-\the\dimexpr\fontdimen22\textfont2\relax ]
  \SetVertexNormal[Shape      = circle, LineWidth  = 1pt]
 \Vertex[x=5 ,y=5,L=a]{2}
   \Vertex[x=5 ,y=2,L=a]{7}
   \Vertex[x=3 ,y=0,L=h]{8}
        \Vertex[x=3 ,y=3,L=h]{10}
     \Vertex[x=3,y=1.5,L=H]{17}
       \Vertex[x=2,y=5,L=a]{19}
 \Vertex[x=2,y=1.5,L=a]{20}
     \SetUpEdge[lw         = 1pt, color      = black, labelcolor = lightgray]      
  \Edge(2)(10)
       \Edge(7)(8)
       \Edge(10)(19)
       \Edge(8)(17) \Edge(17)(10)
       \Edge(8)(20)
        \end{tikzpicture}
& 

 \begin{tikzpicture}[scale=0.6,every node/.style={transform shape}, baseline=-\the\dimexpr\fontdimen22\textfont2\relax ]
  \SetVertexNormal[Shape      = circle, LineWidth  = 1pt]
 \Vertex[x=5 ,y=3,L=a]{2}
     \Vertex[x=5 ,y=0,L=a]{7}
    \Vertex[x=5,y=1.5,L=c]{18}
     \SetUpEdge[lw         = 1pt, color      = black, labelcolor = lightgray]      
 \Edge(2)(18) \Edge(18)(7)
        \end{tikzpicture}

 &  \begin{tikzpicture}[scale=0.6,every node/.style={transform shape}, baseline=-\the\dimexpr\fontdimen22\textfont2\relax ]
  \SetVertexNormal[Shape      = circle, LineWidth  = 1pt]
   \Vertex[x=5 ,y=0,L=a]{2}
   \Vertex[x=9,y=2,L=d]{3}
   \Vertex[x=8,y=0,L=a]{21}
    \Vertex[x=9,y=0,L=a]{22}

     \SetUpEdge[lw         = 1pt, color      = black, labelcolor = lightgray]    
 \Edge(2)(3) 
    \Edge(3)(21)
    \Edge(3)(22) \end{tikzpicture} 
 
 \\ \hline 
 
 \\

$ \varphi_i^\#$ & 

\begin{tikzpicture}[scale=0.6,every node/.style={transform shape}, baseline=-\the\dimexpr\fontdimen22\textfont2\relax ]
  \SetVertexNormal[Shape      = circle, LineWidth  = 1pt]
 \Vertex[x=1 ,y=9,L=b]{1}
   \Vertex[x=5 ,y=7,L=a]{2}
   \Vertex[x=5 ,y=4,L=a]{7}
    \Vertex[x=1 ,y=5,L=b]{9}
       \Vertex[x=0,y=10,L=B]{11}      
      \Vertex[x=0,y=6,L=B]{12}    
         \Vertex[x=2,y=7,L=a]{19}
                  \Vertex[x=2,y=3.5,L=a]{20}

      \SetUpEdge[lw         = 1pt, color      = black, labelcolor = lightgray]      
     \Edge(1)(19)
     \Edge(1)(2)
     \Edge(9)(7)
     \Edge(9)(20)
     \Edge[label={\tiny\textbf{2}}](11)(1)
     \Edge[label={\tiny\textbf{2}}](12)(9)

   \draw[line width=1pt,->] (3,1) -- (3,0) node[pos=.5,right] {{\large $\deg \varphi_1^\# = 4$}};

         \SetVertexNormal[Shape      = circle, LineWidth  = 1pt]
 \Vertex[x=4,y=-1,L=a]{q4}
 \Vertex[x=2,y=-1,L=B]{q5}
 \Vertex[x=3,y=-1,L=b]{q6}
     \SetUpEdge[lw         = 1pt, color      = black, labelcolor = lightgray]      
\Edge(q5)(q6) \Edge(q6)(q4) 

        \end{tikzpicture}

&
        
        \begin{tikzpicture}[scale=0.6,every node/.style={transform shape}, baseline=-\the\dimexpr\fontdimen22\textfont2\relax ]
  \SetVertexNormal[Shape      = circle, LineWidth  = 1pt]
 \Vertex[x=5 ,y=7,L=a]{2}
   \Vertex[x=5 ,y=4,L=a]{7}
   \Vertex[x=3 ,y=2,L=h]{8}
        \Vertex[x=3 ,y=5,L=h]{10}
     \Vertex[x=3,y=3.5,L=H]{17}
       \Vertex[x=2,y=7,L=a]{19}
 \Vertex[x=2,y=3.5,L=a]{20}
     \SetUpEdge[lw         = 1pt, color      = black, labelcolor = lightgray]      
  \Edge(2)(10)
       \Edge(7)(8)
       \Edge(10)(19)
     \Edge[label={\tiny\textbf{2}}](8)(17) \Edge[label={\tiny\textbf{2}}](17)(10){2}
       \Edge(8)(20)
       
    \draw[line width=1pt,->] (3,1) -- (3,0) node[pos=.5,right] {{\large $\deg \varphi_2^\# = 4$}};

         \SetVertexNormal[Shape      = circle, LineWidth  = 1pt]
 \Vertex[x=4,y=-1,L=a]{q4}
 \Vertex[x=2,y=-1,L=H]{q5}
 \Vertex[x=3,y=-1,L=h]{q6}
     \SetUpEdge[lw         = 1pt, color      = black, labelcolor = lightgray]      
\Edge(q5)(q6) \Edge(q6)(q4)

        \end{tikzpicture}

        & 
        
     \begin{tikzpicture}[scale=0.6,every node/.style={transform shape}, baseline=-\the\dimexpr\fontdimen22\textfont2\relax ]
  \SetVertexNormal[Shape      = circle, LineWidth  = 1pt]
 \Vertex[x=5 ,y=5,L=a]{2}
     \Vertex[x=5 ,y=2,L=a]{7}
    \Vertex[x=5,y=3.5,L=c]{18}
    \Vertex[x=6.5,y=3.5,L=$\mbox{ }$]{f}
     \SetUpEdge[lw         = 1pt, color      = black, labelcolor = lightgray]      
 \Edge(2)(18) \Edge(18)(7) \Edge[label={\tiny\textbf{2}}](f)(18)
    
            \draw[line width=1pt,->] (5,1) -- (5,0) node[pos=.5,right] {{\large $\deg \varphi_3^\# = 2$}};
            
         \SetVertexNormal[Shape      = circle, LineWidth  = 1pt]
 \Vertex[x=6.5,y=-1,L=$\mbox{ }$]{q4}
 \Vertex[x=3.5,y=-1,L=a]{q5}
 \Vertex[x=5,y=-1,L=c]{q6}
     \SetUpEdge[lw         = 1pt, color      = black, labelcolor = lightgray]      
\Edge(q5)(q6) \Edge(q6)(q4) 

        \end{tikzpicture}

        &

         \begin{tikzpicture}[scale=0.6,every node/.style={transform shape}, baseline=-\the\dimexpr\fontdimen22\textfont2\relax ]
  \SetVertexNormal[Shape      = circle, LineWidth  = 1pt]
   \Vertex[x=5 ,y=2,L=a]{2}
   \Vertex[x=9,y=4,L=d]{3}
   \Vertex[x=8,y=2,L=a]{21}
    \Vertex[x=9,y=2,L=a]{22}
    \Vertex[x=10,y=5,L=$\mbox{ }$]{ff}

     \SetUpEdge[lw         = 1pt, color      = black, labelcolor = lightgray]    
 \Edge(2)(3) 
    \Edge(3)(21)
    \Edge(3)(22) \
    \Edge[label={\tiny\textbf{3}}](3)(ff)

           \draw[line width=1pt,->] (8,1) -- (8,0) node[pos=.5,right] {{\large $\deg \varphi_4^\# = 3$}};
    
             \SetVertexNormal[Shape      = circle, LineWidth  = 1pt]
 \Vertex[x=9.5,y=-1,L=$\mbox{ }$]{q4}
 \Vertex[x=6.5,y=-1,L=a]{q5}
 \Vertex[x=8,y=-1,L=d]{q6}
     \SetUpEdge[lw         = 1pt, color      = black, labelcolor = lightgray]      
\Edge(q5)(q6) \Edge(q6)(q4)

    \end{tikzpicture}

\\ \hline
\end{tabular}
\end{table}
}

\noindent \textbf{The construction of $G^\#$ (continued).}\ 
 Define, for any $v\in\f^{-1}(x_0)\cap \V(G^s)$, the integer
\[d^\#(v):=\sum_{i\in I_L}\left(\sum_{e\in\E_v(G^s)\cap \E(G_i^\#)}r_{\f_i^\#}(e)\right)-\sum_{i\in I_R}\left(\sum_{e\in\E_v(G^s)\cap \E(G_i^\#)}r_{\f_i^\#}(e)\right).\]

 The graph $G^\#$ is obtained by gluing all $G^\#_i$ together at $\iota_i(v)$ for all $v\in\f^{-1}(x_0)\cap G_i''$, and, for any $v$ with $d^\#(v) \neq 0$, 
gluing an additional copy $(S^{\#,v},x_v)$ of $(S^\#,x)$ at $v$.\\

\begin{footnotesize}
\noindent \textbf{Example (continued).}\  In the example, the vertices $v \in \varphi^{-1}(x_0)$ are the numbered vertices in the following display of $G^s$: 
\begin{center}
\begin{tikzpicture}[scale=0.6,every node/.style={transform shape}]
\SetVertexNormal[Shape      = circle,
                 FillColor  = black,
                 LineWidth  = 1pt]
                 
              \Vertex[x=1 ,y=9,L=b]{1}

   \Vertex[x=9,y=9,L=d]{3}
   \Vertex[x=7 ,y=5,L=f]{4}
   \Vertex[x=7 ,y=2,L=g]{5}
   \Vertex[x=9 ,y=5,L=e]{6}

   \Vertex[x=3 ,y=2,L=h]{8}
    \Vertex[x=1 ,y=5,L=b]{9}
     \Vertex[x=3 ,y=5,L=h]{10}
         \Vertex[x=3,y=3.5,L=$\mbox{ }$]{17}
      \Vertex[x=5,y=5.5,L=$\mbox{ }$]{18}

     \SetVertexNormal[Shape      = circle,
                 FillColor  = white,
                 LineWidth  = 1pt]
          \Vertex[x=2,y=7,L=1]{19}
             \Vertex[x=5 ,y=7,L=2]{2}
               \Vertex[x=2,y=3.5,L=5]{20}
                        \Vertex[x=8,y=7,L=3]{21}
                              \Vertex[x=9,y=7,L=4]{22}
                                    \Vertex[x=8,y=3.5,L=8]{23}
                                          \Vertex[x=7,y=3.5,L=7]{24}
                                             \Vertex[x=5 ,y=4,L=6]{7}

 \SetUpEdge[lw         = 1pt, color      = black ]

   \Edge(1)(2)
   \Edge(2)(10)
    \Edge(2)(4)
    \Edge(2)(3)
     \Edge(5)(7)
     \Edge(6)(7)
     \Edge(7)(8)
     \Edge(7)(9)      
          
    \Edge(1)(19) \Edge(19)(10)
    \Edge(2)(18) \Edge(18)(7)
    \Edge(3)(21) \Edge(21)(4)
    \Edge(3)(22) \Edge(22)(6)
    \Edge(5)(23) \Edge(23)(6)
    \Edge(8)(20) \Edge(20)(9)
    \Edge(4)(24) \Edge(24)(5)
    \Edge(8)(17) \Edge(17)(10)
    \end{tikzpicture}
\end{center}
The vertices labeled 1 and 5 have only neighbouring subgraphs $G_i^\#$ with $i \in I_L$, and the indices are all one, and add up to $2$. The vertices labeled 3, 4, 7 and 8 have only neighbouring subgraphs $G_i^\#$ with $i \in I_R$, and the indices are all one, and add up to $2$. Finally, the vertices labeled 2 and 6  have neighbouring subgraph $G_2^\#$ with index in $I_L$,  and neighbouring subgraphs $G_3^\#$ and two of $G_i^\#$ (with $i=4,5,6,7$) with index in $I_R$, for which all the indices are all one. Thus, 
 $$ d^\#(v) = \left\{ \begin{array}{ll} 1+1=2 & \mbox{ if } v \mbox{ has label 1 or 5}; \\ 1+1-(1+1+1)=-1 & \mbox{ if } v \mbox{ has label 2 or 6}; \\ -(1+1) = -2 & \mbox{ if } v \mbox{ has label 3,4, 7 or 8}. \end{array}\right. $$
Thus, we need to glue in one extra copy of $S^\#$ at every such vertex. $\Diamond$
\end{footnotesize}\\

\noindent \textbf{The construction of $\f^\#$.}\ 
To define $\f^\#$, it suffices to define its restriction to $G_i^\#$ for $i=1,\dots,d$, and its restriction to $S^{\#,v}$, and show that these are compatible on intersections.  Define the restrictions as follows: 
\begin{enumerate}
\item If $i \in I_L$, set $\f^\#|_{G_i^\#} \colon G_i^\# \xrightarrow{\f_i^\#}S^\# \isomto T^\#_L$ with index $r_{\f^\#}(e)=r_{\f^\#_i}(e)$ for all $e \in G_i^\#$; 
\item If $i \in I_R$, set $\f^\#|_{G_i^\#} \colon G_i^\# \xrightarrow{\f_i^\#}S^\# \isomto T^\#_R$ with index $r_{\f^\#}(e)=r_{\f^\#_i}(e)$ for all $e \in G_i^\#$;
\item If $v \in \f^{-1}(x_0)$ with $d^\#(v) > 0$, set $\f^\#|_{S^{\#,v}} \colon S^{\#,v} \isomto T^\#_R$ with index $r_{\f^\#}(e)=d^\#(v)$ for all $e \in S^{\#,v}$; 
\item If $v \in \f^{-1}(x_0)$ with $d^\#(v) < 0$, set $\f^\#|_{S^{\#,v}} \colon S^{\#,v} \isomto T^\#_L$ with index $r_{\f^\#}(e)=-d^\#(v)$ for all $e \in S^{\#,v}$.  
\end{enumerate}
One checks that this glues together correctly to a finite graph morphism $\f^\# \colon G^\# \rightarrow T^\#$. \\

\begin{footnotesize}
\noindent \textbf{Example (continued).}\ The extra copies of $S^\#$ that are glued to vertices 1, 3, 4, 5, 7, 8 get index 2 on every edge, but the copies that are glued to 2 and 6 get index one.   The final re-engineered map $\varphi^\# \colon G^\# \rightarrow T^\#$ in our example is given as follows: 

\begin{center}
\begin{tikzpicture}[scale=0.7,every node/.style={transform shape}]
\SetVertexNormal[Shape      = circle, LineWidth  = 2pt]
   \Vertex[x=1 ,y=9,L=L]{1}
   \Vertex[x=5 ,y=7,L=a]{2}
   \Vertex[x=9,y=9,L=R]{3}
   \Vertex[x=7 ,y=5,L=R]{4}
   \Vertex[x=7 ,y=2,L=R]{5}
   \Vertex[x=9 ,y=5,L=R]{6}
   \Vertex[x=5 ,y=4,L=a]{7}
   \Vertex[x=3 ,y=2,L=L]{8}
    \Vertex[x=1 ,y=5,L=L]{9}
     \Vertex[x=3 ,y=5,L=L]{10}
     
      \SetVertexNormal[Shape      =circle, LineWidth  = 1pt]
                 \renewcommand{\VertexInterMinSize}{10pt}  
   \Vertex[x=0,y=10,L=l]{11}      
      \Vertex[x=0,y=6,L=l]{12}              
   \Vertex[x=5,y=8,L=L]{13}  
     \Vertex[x=5,y=9,L=l]{14a}         
   \Vertex[x=5,y=3,L=L]{15}  
      \Vertex[x=5,y=2,L=l]{16a}  
      \Vertex[x=3,y=3.5,L=l]{17}
      \Vertex[x=5,y=5.5,L=R]{18}
        \Vertex[x=6,y=5.5,L=r]{18a}
            \Vertex[x=2,y=7,L=a]{19}
            \Vertex[x=1,y=7,L={\tiny R}]{19a}
              \Vertex[x=0,y=7,L={\tiny r}]{19b}
                  \Vertex[x=2,y=3.5,L=a]{20}
                    \Vertex[x=1,y=3.5,L={\tiny R}]{20a}
                      \Vertex[x=0,y=3.5,L={\tiny r}]{20b}
                        \Vertex[x=8,y=7,L=a]{21}
                          \Vertex[x=7,y=7,L={\tiny L}]{21a}    
                           \Vertex[x=6,y=7,L={\tiny l}]{21b}     
                            \Vertex[x=9,y=7,L=a]{22}
                              \Vertex[x=10,y=7,L={\tiny L}]{22a}     
  \Vertex[x=11,y=7,L={\tiny l}]{22b}     

                                    \Vertex[x=8,y=3.5,L=a]{23}
                                      \Vertex[x=9,y=3.5,L={\tiny L}]{23a}     
  \Vertex[x=10,y=3.5,L={\tiny l}]{23b}  
                                          
\Vertex[x=7,y=2.7,L=a]{24}
\Vertex[x=7.5,y=3.8,L={\tiny L}]{24a}     
\Vertex[x=8.3,y=4.5,L={\tiny l}]{24b}  

\Vertex[x=10,y=10,L=r]{25}
\Vertex[x=8,y=1,L=r]{26}
 
\draw[line width=1pt,->] (5,1.5) -- (5,-0.5) node[pos=.5,right] {{\Large $\varphi^\#$}};
             
 
 \SetVertexNormal[Shape      = circle, LineWidth  = 1pt]

\Vertex[x=5,y=-1,L=a]{center}
 \Vertex[x=4,y=-1,L=L]{links1}
 \Vertex[x=3,y=-1,L=l]{links2}
 \Vertex[x=6,y=-1,L=R]{rechts1}
  \Vertex[x=7,y=-1,L=r]{rechts2}
  \SetUpEdge[lw         = 2pt,
           color      = black]

   \Edge(1)(2)
   \Edge(2)(10)
    \Edge(2)(4)
    \Edge(2)(3)
     \Edge(5)(7)
     \Edge(6)(7)
     \Edge(7)(8)
     \Edge(7)(9)      
 
\SetUpEdge[lw         = 2pt,
           color      = black ]

    \Edge(1)(19) \Edge(19)(10)
    \Edge(2)(18) \Edge(18)(7)
    \Edge(3)(21) \Edge(21)(4)
    \Edge(3)(22) \Edge(22)(6)
    \Edge(5)(23) \Edge(23)(6)
    \Edge(8)(20) \Edge(20)(9)
    \Edge(4)(24) \Edge(24)(5)
    \Edge(8)(17) \Edge(17)(10){2}

    \SetUpEdge[lw         = 1pt,
           color      = black,
           labelcolor = lightgray]              
  \Edge[label={\tiny\textbf{2}}](1)(11)
    \Edge[label={\tiny\textbf{2}}](9)(12)
    \Edge(2)(13) \Edge(13)(14a) 
    \Edge(7)(15) \Edge(15)(16a)
    
    \Edge[label={\tiny\textbf{2}}](19)(19a) \Edge[label={\tiny\textbf{2}}](19a)(19b)
       \Edge[label={\tiny\textbf{2}}](20)(20a) \Edge[label={\tiny\textbf{2}}](20a)(20b)
        \Edge[label={\tiny\textbf{2}}](21)(21a) \Edge[label={\tiny\textbf{2}}](21a)(21b)
         \Edge[label={\tiny\textbf{2}}](22)(22a) \Edge[label={\tiny\textbf{2}}](22a)(22b)
          \Edge[label={\tiny\textbf{2}}](23)(23a) \Edge[label={\tiny\textbf{2}}](23a)(23b)
           \Edge[label={\tiny\textbf{2}}](24)(24a) \Edge[label={\tiny\textbf{2}}](24a)(24b)
     \Edge[label={\tiny\textbf{2}}](18)(18a)
     
  \Edge[label={\tiny\textbf{3}}](3)(25)
    \Edge[label={\tiny\textbf{3}}](5)(26)
    
       \Edge(center)(links1) \Edge(links1)(links2) \Edge(center)(rechts1) \Edge(rechts1)(rechts2)

              \end{tikzpicture}
\end{center}
Here, the vertex $X_0$ is labelled ``a''. The map $\varphi$ has degree $18$ and the edges $\{L,a\}$ and $\{a,R\}$ of $T^\#$ both have size $2/5>A/2=1/10$.  $\Diamond$ 
\end{footnotesize} \\

\noindent\textbf{The finite morphism $\f^\#$ is harmonic.}\ 
We check that  $m_{\f^\#}(v)$ is well-defined for all $v\in\V(G^\#)$. For all $v\not\in\f^{\#-1}(X_0)$ there is either a unique $i=1,\dots,d$ such that $v\in\V(G^\#_i)$ and then $m_{\f^\#}(v)=m_{\f_i''}(v)$, or there is a unique $w\in\f^{\#-1}(X_0)$ such that $v\in\V(S^{\#,x_w})$ and then $m_{\f^\#}(v)=|d_{\f^\#}(w)|$. For all $v\in\f^{\#-1}(X_0)$, it holds that 
\[m_{\f^\#}(v)=\max\left\{\sum_{i\in I_L}\left(\sum_{e\in\E_v(G^\#)\cap \E(G_i^\#)}r_{\f^\#}(e)\right),\sum_{i\in I_R}\left(\sum_{e\in\E_v(G^\#)\cap \E(G_i^\#)}r_{\f^\#}(e)\right)\right\}.\]

\noindent\textbf{The edge $e^\#$.}\
 Any of the two edges $e^\#\in\E_{X_0}(T^\#)$ satisfies
\[\size_{\varphi^\#_*\mu_G}(e^\#)\geq\min\left\{\nu(\bigcup_{i\in I_L}T''_i),\nu(\bigcup_{i\in I_R}T''_i)\right\}>\frac{A}{2}.\]

\noindent \textbf{The degree of $\f^\#$.}\
Consider a vertex $v \in \V(G^\#)\cap\f^{-1}(x_0)$: if $v$ belongs to $G^s$, then it belongs to at most $\Delta_G$ different $G^\#_i$ for $i=1,\dots,d$; and if $v \notin G^s$, then there is a unique $i=1,\dots,d$ such that the unique path from $v$ to $G^s$ is contained in 
$G''_i$.
 It follows that for each $v\in\V(G^\#)\cap\f^{-1}(x_0)$, at most $\Delta_G$ of the neighboring $G_i^\#$ are either all sent to $T^\#_L$, or all to $T^\#_R$. Hence $$m_{\f^\#}(v)\leq \Delta_G m_\f(v),$$ and this implies $\deg \f^\# \leq \Delta_G \deg \f$. 
\end{proof}

\begin{remark}
 The point $x_0$ (used in the proof) with the property that all components of $T-x_0$ have measure $<C/2$ is in fact \emph{unique}. Indeed, if there are two such vertices, say, $x_0$ and $x_1$, then let $e=(x,y)$ denote any edge on a path between $x_0$ and $x_1$. One component of $T-x$ contains $T_1(e)$ and one component of $T-y$ contains $T_2(e)$, and hence by assumption, $\nu(T_1(e))<C/2$ and $\nu(T_2(e))<C/2$ but $\nu(T_1(e))+\nu(T_2(e))=1$. Hence $C>1$, but this is impossible with $C\leq 1-A-B$ and $A,B>0$.
\end{remark}

\section{Discussion of the spectral lower bound on stable gonality} 
We now give some examples that illustrate the bound. 

\begin{example}\label{bananensplit}
For the banana graph $B_n$, we have $\Delta_{B_n}=n$, $|B_n|=2$ and $\lambda_{B_n}=2n$, so the lower bound is trivial: $\mathrm{sgon}(B_n) \geq 1$. However, the stable gonality of $B_n$ (for $n \geq 2$), equals $2$. See Figure \ref{figbanana} for such a map of degree $2$. 
\end{example}

\begin{example} \label{ASM}
For the complete bipartite graph $K_{n,n}$, we have $\Delta_{K_{n,n}}=n$, $|K_{n,n}|=2n$ and $\lambda_{K_{n,n}}=n$. If $n$ is even, then the lower bound is $$\mathrm{sgon}(K_{n,n}) \geq \left\lceil \frac{2n^2}{5n+4} \right\rceil,$$

We expect that the stable gonality of $K_{n,n}$ equals $n$. 
A morphism which attains degree $n$ is given by mapping $K_{n,n}$ to the star with one central vertex and $n$ emanating edges in the obvious way. For $n=p^r+1$ ($p$ prime), the graph $K_{n,n}$ occurs as stable reduction graph of the curve $$ X_{\lambda,r} \colon (x^{p^r}-x)(y^{p^r}-y) =  \lambda$$ (seen in $\PP^1 \times \PP^1$) with $|\lambda|<1$, over a valued field $(k,|\cdot|)$ of characteristic $p$, which, as a fiber product of two projective lines,  admits an obvious morphism of degree $p^r+1$ to $\PP^1$.  The stable reduction itself consist of two transversally intersecting families of $p^r+1$ rational curves (``check board with $p^{2r}$ squares''). For more details on these curves, see for example \cite{CKK}. 
\end{example}

In the family of Example \ref{ASM}, our lower bound has the same order of growth in $n$ as does the expected gonality, but is about 5 times as small. This seems to be a general phenomenon; we don't know an interesting example where our lower bound is sharp. 

\begin{example} \label{rama}
If $X_{n}$ is a family of Ramanujan graphs with (fixed) regularity $d$ and $n$ vertices ($n$ increasing), then by the Alon-Boppana bound, we get inequalities $$\sqrt{d-1}-o(1) \leq \lambda_{X_n}/2 \leq  \sqrt{d-1},$$ so we find a lower bound of the form 
$$  \mathrm{sgon}(X_n) \geq  \kappa_d \cdot n $$
for $n$ sufficiently large with $\kappa_d$ a constant only depending on $d$. In any family of Mumford curves whose stable reduction graphs are $d$-regular Ramanujan graphs, the gonality goes to infinity as the number of components of the stable reduction does so. 
\end{example}

\begin{remark} \label{LYfout} The famous Li--Yau inequality from differential geometry \cite{LiYau} states that the gonality $\mathrm{gon}(X)$ of a compact Riemann surface $X$ (minimal degree of a conformal mapping $\varphi$ of $X$ to the Riemann sphere) is bounded below by
$$ \mathrm{gon}(X) \geq \frac{1}{8 \pi} \lambda_X \mathrm{vol}(X), $$
where $\lambda_X$ is the first non-trivial eigenvalue of the Laplace-Beltrami operator of $X$, and $\mathrm{vol}(X)$ denotes the volume of $X$.  

For graphs $G$ with any Laplacian (normalized or not),  an inequality of the form 
$$ \text{``}\mathrm{sgon}(G) \geq \kappa \cdot \lambda_G \cdot \mathrm{vol}(G)\text{''}  \ \ \ (\ast) $$
for some constant $\kappa$ fails. 
A counterexample is given by the complete graph $K_{n}$, which has stable gonality $n-1$. However, a lower bound of the form $(\ast)$ would be $\kappa \cdot  n^2(n-1) $ for the usual Laplacian, and $\kappa \cdot n^2$ for the normalized Laplacian (see Table \ref{table} in the appendix for the data; one deduces that the analog of the Li--Yau inequality also fails if one uses any of the other notions of gonality from the existing literature and are outlined in the appendix.) 

One sees from our result that in a graph, the constant $\kappa$ needs to be roughly divided by the maximal edge degree for such an inequality to hold. 

\end{remark}

As we have seen in Corollary \ref{steq}, stable gonality is defined on equivalence classes of graphs, in the sense that two graphs $G$ and $G'$ are equivalent (notation $G \sim G'$) if they are refinements of the same stable graph. Hence the result also implies that 

\begin{corollary} For any graph $G$ with $g \geq 2$, we have 
$$\mathrm{sgon}(G) \geq \max_{G' \sim G} \left\lceil \frac{\lambda_{G'}}{\lambda_{G'}+4(\Delta_{G'}+1)} |G'| \right\rceil. $$ \hfill $\Box$ \end{corollary}

\begin{remark}
It is tempting to consider the limiting value for the lower bound in this theorem when the graph is further and further refined. Whereas it is clear how the number of vertices and the maximal vertex degree change under refinements, the change of the eigenvalue under refinements is not so well-understood (apart from regular graphs). For applications in solid state physics, Eichinger and Martin have developed an algorithm that computes the change in eigenvalues under refinement by applying only linear algebra to the original Laplace matrix \cite{EM}. Examples (such as the banana graph) suggest that (iterated)  refinement might worsen the lower bound. 
\end{remark}

There is a similar result for the normalized Laplacian. Denote with $\tilde \lambda_G$ the first non-trivial eigenvalue of $L^\sim_G$.

\begin{theorem}\label{genormaliseerdegrens} Let $G$ be a graph with maximal degree $\Delta_G$ and first normalized Laplace eigenvalue $\tilde{\lambda}_G$. The stable gonality of $G$ is bounded from below by
\[\sgon(G)\geq \frac{\tilde \lambda_G}{\Delta_G\tilde\lambda_G+4(\Delta_G+1)}\vol(G).\] 
\end{theorem}

\begin{proof}[Sketch of proof] The proof is virtually the same as Theorem \ref{spectral}, so instead of providing all details, we briefly outline the differences. Instead of $\mu_G$, we use the measure $\eta_G$ on $V(G')$ defined for $A\subset V(G')$ by  
\[\eta_G(A):=\frac{\sum\limits_{v\in A\cap G}d_{v}^G}{\vol(G)},\] 
where $d_v^G$ is the degree of $v$ in $G$. Lemma \ref{Amini} is valid for all probability measures, and therefore also for $\f_*\eta_G$. Since $\f_*\eta_G(x)$ counts the number of edges instead of vertices, the conclusion of Lemma \ref{trivialcase} changes to $$\deg\f\geq\frac{\f_*\eta_G(x)}{\Delta_G}\vol(G).$$ The analogue of Proposition \ref{Cheeger} can be derived by using the test function 
\[f(v)=\left\{\begin{array}{cl}\frac{1}{\vol(G_1)}&\text{if } v\in G_1,\\-\frac{1}{\vol(G_2)}&\text{if } v\in G_2.\end{array}\right.\]
Proposition \ref{change} does not change for this new measure. 
We conclude that the proof of Corollary \ref{deafschatting} only changes in the step where Lemma \ref{trivialcase} is used.
\end{proof}

\begin{remark} For $k$-regular graphs $$k\tilde\lambda_G=\lambda_G \mbox{ and }\vol(G)/k=|G|,$$ and hence the lower bounds are identical for the two different Laplace operators. 
\end{remark}

\begin{remark}  Generically (in the sense of algebraic geometry), the gonality of a curve attains the Brill--Noether bound (cf.\ for example the references in the Appendix of \cite{Poonen}). However, curves of fixed genus and fixed stable reduction graph can have widely varying gonality (e.g., the banana graph $B_n$ has stable gonality $2$, but its subdivisions can occur as stable reduction graph of curves whose gonality takes on all the values $2,\dots,n$); in particular, the gonality of the curve can be much higher than the stable gonality of the reduction graph. One may try to find the most probable stable gonality of a random connected (multi-)graph and compare it to the most probable value of the Brill--Noether bound. In this remark, we compute something much simpler: the difference between the expected value of the Brill--Noether bound and the lower bound in our theorem, for the Erd\H{o}s--R\'enyi random graph model with a specific connection probability in the non-sparse region.   

For a random graph model $G=G_{n,p}$ of Erd\H{o}s-R\'enyi type \cite{Erdoes} with $n$ vertices and edge probability $p=p(n)=n^{-\delta}$ for some $0<\delta<1$,  the threshold for almost sure connectivity holds, the expected number of edges is $n(n-1)p/2$, so the first Betti number  of $G_{p,n}$ is $g = \frac{1}{2}n^{1-\delta}(n-1)-n+1$ almost surely. Chung, Lu and Vu \cite{CLV} have shown that (for a class of function including these $p(n)$) the normalized eigenvalue $\tilde\lambda$ tends to $1$ with high probability. 
Also, the given assumptions imply that $\Delta_{G_{n,p}}=pn(1+o(1))$ in probability (\cite{B}, 3.14). 

Hence the lower bound tends with high probability to $$\approx \frac{n}{5} \approx \frac{1}{5} \sqrt[2-\delta]{2g}, $$ 
 which is sublinear in $g$ (and for $\delta \rightarrow 0$, tends to $\sqrt{2g}/5$, up to a constant the actual value of the stable gonality $n-1 = \sqrt{2g}$ for the complete graph $K_n$ of genus $g=(n-1)^2/2$), whereas the Brill--Noether bound is linear in $g$ (which happens if $\delta \rightarrow 1$). 

There are at least two ways to interpret this heuristic observation: either the lower bound is asymptotically bad for random graphs; or stable gonality of random graphs is significantly lower than generic gonality of curves. 
\end{remark}

\section{A linear lower bound on the gonality of Drinfeld modular curves: proof of Theorem \ref{DrinGon}}

We recall the main concepts and notations from the theory of general Drinfeld modular curves, cf.\  \cite{GekelerLN}, \cite{GekelerReversat}. 

\begin{se}
Let $K$ denote a global function field of a smooth projective curve $X$ over a finite field $k=\F_q$ with $q$ elements and characteristic $p>0$, and $\infty$ a place of degree $\delta$ of $K$. Let $\pi_\infty$ denote a uniformizer at $\infty$. Let $A$ denote the subring of $K$ of elements that are regular outside $\infty$. 
\end{se} 

\begin{se} Let $Y$ denote a rank-two $A$-lattice in the completion $K_\infty$ of $K$ at $\infty$. Such lattices are classified up to isomorphism by their determinant, so they are isomorphic to $A \oplus I$, where $I$ runs through a set of representatives of $\mbox{Pic}(A)$, the ideal class group of $A$. 

Let $\GL(Y)$ denote the automorphism group of the lattice $Y$: $$\GL(Y)=\{\gamma \in \GL_2(K) \colon \gamma Y = Y\},$$ and let $\Gamma$ denote a congruence subgroup of $\Gamma(Y):=\GL(Y)$. This means that $\Gamma$ contains a principal congruence group $\Gamma(Y,\fn)$ as a finite index subgroup, where $$\Gamma(Y,\fn) = \ker \left( \Gamma(Y) \rightarrow \GL(Y/\fn Y) \right),$$ for $\fn$ an ideal in $A$. Let $Z \cong \F_q^*$ denote the center of $\GL(Y)$. 

If $ Y= A \oplus A$ is the ``standard'' lattice, we revert to the standard notations $\Gamma(1):=\Gamma(A \oplus A)$ and $\Gamma(\fn):=\Gamma(A \oplus A, \fn)$. 
\end{se}

\begin{se}
The groups $\Gamma$ act by fractional transformations on the Drinfeld space $\Omega = \C_\infty - K_\infty$, where $\C_\infty$ is the completion of an algebraic closure of $K_\infty$. The quotient $\Gamma \backslash \Omega$ is an analytic smooth one-dimensional space, and is the analytification of a smooth affine algebraic curve $Y_\Gamma$, that can be defined over a finite abelian extension of $K$ inside $K_\infty$. It can be compactified to a Drinfeld modular curve $X_\Gamma$ by adding finitely many points, called cusps. 

The $\C_\infty$-points of the (coarse) moduli scheme $M(\fn)$ of rank-two Drinfeld $A$-modules with full level $\fn$-structure (i.e., an isomorphism of $(A/\fn)^2$ with the torsion of the Drinfeld module) can be described as 
$$ M(\fn)(\C_\infty) = \bigsqcup_{Y \in \mathrm{Pic}(A)} \Gamma(Y,\fn) \backslash \Omega.$$
We denote such a component by $Y(Y,\fn):=\Gamma(Y,\fn) \backslash \Omega$, and its compactification by $X(Y,\fn)$. 
\end{se} 

\begin{theorem}[= Theorem \ref{DrinGon}] \label{DrinGon2}
Let $\Gamma$ denote a congruence subgroup of $\Gamma(Y)$. Then the gonality of the corresponding Drinfeld modular curve $X_\Gamma$ satisfies 
$$ \mathrm{gon}_{\bar{K}}(X_{\Gamma}) \geq c_{q,\delta} \cdot [\Gamma(Y):\Gamma] $$
where the constant $c_{q,\delta}$ is 
\[c_{q,\delta}:= \frac{q^\delta-2\sqrt{q^\delta}}{5q^\delta-2\sqrt{q^\delta}+8}\cdot\frac{1}{q(q^2-1)}\]
This implies a linear lower bound in the genus of modular curves of the form
$$ \mathrm{gon}_{\bar{K}}(X_{\Gamma}) \geq c'_{K,\delta} \cdot (g(X_\Gamma) -1 ),$$
where $c_{K,\delta}$ is a bound that depends only on the function field $K$. If $K$ is a rational function field and $\delta=1$, then we can put $c'_{K,\delta}=2c_{q,1}$. 
\end{theorem}

\begin{proof} First observe that $\mathrm{gon}_{\bar{K}}(X) = \mathrm{gon}_{\bar{K_\infty}}(X)$, so we now consider $X_\Gamma$ as a curve over $k=K_\infty$ and are in a set up where we can apply our previous results. The remainder of the proof has various parts. 

\medskip

\noindent  \textbf{Reduction to principal congruence groups.}\  First of all, we observe that it suffices to prove the bound for the groups $\Gamma(Y,\fn)$. Indeed, if $\varphi \colon X_\Gamma \rightarrow \PP^1$ is a morphism, then from the inclusion $\Gamma(Y,\fn) \leq \Gamma$ we get a composed morphism 
\begin{equation} \label{princov} X_{\Gamma(Y,\fn)} \rightarrow  X_\Gamma \rightarrow \PP^1 \end{equation}
of degree $$\frac{[\Gamma:\Gamma(Y,\fn)]}{|\Gamma \cap Z|} \cdot \deg \varphi,$$ and hence
\begin{equation} \label{prinineq} \mathrm{gon}_{\bar{K}}(X_\Gamma) \geq \mathrm{gon}_{\bar{K}}(X(Y,\fn)) / [\Gamma:\Gamma(Y,\fn)]. \end{equation}
Therefore, the desired inequality $$\mathrm{gon}_{\bar{K}} (X_\Gamma) \geq c_q [\Gamma(Y):\Gamma]$$ follows from $$\mathrm{gon}_{\bar{K}}(X(Y,\fn)) \geq c_q [\Gamma(Y):\Gamma(Y,\fn)].$$

\medskip

We now prove the gonality bound by invoking Theorem \ref{spectral} for the reduction graph of the Drinfeld modular curve $X(Y,\fn)$. 

\medskip 

\noindent \textbf{Semistable model.} First, we construct a semi-stable model for the reduction of $X(Y,\fn)$ at $\infty$. The groups $\Gamma=\Gamma(Y,\fn)$ also act by automorphisms on the Bruhat--Tits tree $\T$ of $\PGL(2,K_\infty)$ \cite{Serre}. The quotient $\Gamma \backslash \T$ is the union of a finite graph $(\Gamma \backslash \T)^0$ and a finite number of half lines in correspondence with the cusps of $X_\Gamma$, and the curve $X_\Gamma$ is a Mumford curve over $K_\infty$ \cite{Mumford} such that the intersection dual graph of the reduction, which is a 
finite union of rational curves over $\F_{q^\delta}$ intersecting transversally in $\F_{q^\delta}$-rational points, equals the finite graph $(\Gamma \backslash \T)^0$ (\cite{GekelerReversat} (2.7.8)). In particular, the genus of the modular curve $X_\Gamma$ equals the first Betti number of this graph (compare \cite{GekelerLN} V.A.11).

Now consider the $\Gamma$-\emph{stable} part $\T^s$ of $\T$, defined to consist of those vertices and edges of $\T$ that have trivial stabilizer for the action of $\Gamma$. Since the stabilizers of cusps are non-trivial and stabilizers of edges are subgroups of stabilizers of adjacent vertices, the stable part is a tree that ends in half-edges (i.e., edges with only an initial vertex). Let $\T^{ss}$ denote the tree obtained from deleting the half-edges, and call the images in $\Gamma \backslash \T^{ss}$ of the remaining vertices that were incident to the half-lines the \emph{boundary points}; say there are $h_c$ of those, see Figure \ref{drinfeld}. 

\begin{center}
\begin{figure}[h]
\includegraphics[width=10cm]{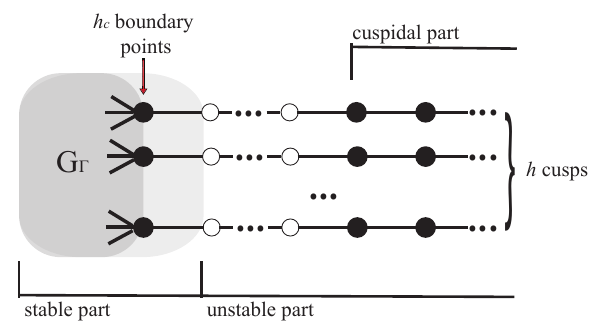}
\caption{Schematic depiction of the quotient graph $\Gamma \backslash \T$, including stable, unstable and cuspidal part.}
\label{drinfeld}
\end{figure}
\end{center}

We claim that the quotient $$G_\Gamma:=\Gamma \backslash \T^{ss}$$ is a semistable reduction graph for $X_\Gamma$. Since it is a subgraph of $(\Gamma \backslash \T)^0$, it suffices to check that these graphs have the same genus. This can be seen as follows. Since $\Gamma=\Gamma(Y,\fn)$ is $p'$-torsion free, \cite{Serre} II.2.9, Ex.\ 2b and Thm.\ 13'(c) imply that $$| \V(\Gamma \backslash \T^s)| - |\E(\Gamma \backslash \T^s)| = \chi(\Gamma) = 1 - g(X_\Gamma) - h,$$ where $h$ is the number of cusps of $\Gamma$ and $\chi(G)$ is the Euler-Poincar\'e characteristic of $G$. Since $h_c = |\E(\Gamma\backslash \T^{s})|-|\E(\Gamma \backslash\T^{ss})|$ and $|\V(\Gamma\backslash \T^{ss})|=|\V(\Gamma \backslash\T^s)|$, we find 
$$ g(G_\Gamma) = g(X_\Gamma)+h-h_c.$$
Now observe that $h \geq h_c$; indeed, $h$ is the number of half-lines of $\Gamma \backslash \T$. A priori several half-lines could be attached to the same boundary point. On the other hand, all paths in the unstable graph attached to boundary vertices have to be part of infinite half-lines (so correspond to cusps). Indeed, if there would be such a \emph{finite} path $P$, let $\tilde{P}$ denote a connected lift of $P$ to $\T$ and $e_P$ any half-line of $\T$ that contains $\tilde{P}$. Then the projection of $e_P$ in $\Gamma \backslash \T$ has to be infinite, since the orders of the stabilizers are strictly increasing along $e_P$, and hence it would have to intersect the stable part $\Gamma \backslash \T^s$. This is impossible, since in the stable part, the stabilizers are trivial. 

Since also $g(G_\Gamma) \leq g(X_\Gamma)$, we conclude that  $h_c=h$ and $g(G_\Gamma) = g(X_\Gamma)$. This means that $\Gamma \backslash \T$ is $G_\Gamma$ connected via $h$ paths to the $h$ cusps. 

If not indicated otherwise, choose $\fn \neq 1$ and write $G:=G_{\Gamma}$.

\medskip

\noindent \textbf{A lower bound on the number of vertices.}\ In the Bruhat--Tits tree $\T$ of $\PGL(2,K_\infty)$, every vertex is ($q^\delta+1$)-regular. Let us consider the special vertex of $\T$ corresponding to the class of the trivial rank-two vector bundle $[\mathcal{O}_\infty \oplus \mathcal{O}_\infty]$ on $X$, and let $v_0$ denote the corresponding vertex in $\Gamma(Y)\backslash \T$. The stabilizer of this vertex is precisely $\PGL(2,\F_{q^\delta})$ (namely, an element of the stabilizer induces an automorphism of the ``star'' of the vertex, which is given by $\PP^1(\F_{q^\delta})$.) The stabilizer intersects $\Gamma(Y)/Z$ (where $Z$ is the center) in the ``constant group'' $\PGL(2,\F_{q})$, and the group $\Gamma=\Gamma(Y,\fn)$ (for $\fn \neq 1$) in the trivial group. There is a covering  map $$\Gamma \backslash \T \rightarrow \Gamma(Y)\backslash \T.$$ 
We conclude that 
\begin{equation} \label{een} |G| \geq \frac{1}{q(q^2-1)} \cdot [\Gamma(Y):\Gamma],  \end{equation} 
since the right hand side is the number of vertices in $\Gamma \backslash \T$ above $v_0$ ---which are stable, since they have trivial stabilizers in $\Gamma$---, and $\PGL(2,\F_q)$ has cardinality $q(q^2-1)$.

\begin{remark}
This estimate for the number of vertices of $G$ will be enough for our purposes, since it differs from the index only by a constant in $q$. But one might also count the total number of vertices of the graph. For a rational function field $K=\F_q(T)$ with a place $\infty$ of degree one, this is easily done, the result being
$$ |G_{\Gamma(\fn)}| = \frac{2q^{\deg(\fn)+1}-q-1}{q^{\deg(\fn)+1}(q^2-1)(q-1)}[\Gamma(1):\Gamma(\fn)];  $$
compare also with computations in \cite{Morgenstern} (cf.\ \cite{Carbone}, \cite{Rust}) and \cite{GekelerNonnengardt}. It seems another proof of the lower bound on the gonality is possible by using Morgenstern's result that there is a perfect matching between a very large (constant fraction depending only on $q$, not on $\deg(\fn)$) subset of the vertices above $v_0$ in $G_{\Gamma(\fn)}$ and vertices in the complement, but we did not pursue this, since it would give a less general and worse result.  
\end{remark}

\begin{remark}
The gonality is \emph{not} always realized by the obvious map $X_\Gamma \rightarrow X(1) \cong \PP^1$. For example, set $q=2$ and let $\p$ denote an prime of degree $3$; then the modular curve $X_0(\p)$ is hyperelliptic, but the map $X_0(\p) \rightarrow X(1)$ has degree $9$. Also notice that for a general base field $K$, the modular curve $X(1)$ is not even itself a rational curve. 
\end{remark} 

\begin{remark} 
Counting the number of cusps (so the number of vertices above a vertex in $\Gamma(Y) \backslash \T$ corresponding to a split bundle of high degree) is not enough to get a linear estimate in the index, since the cusps have rather large stabilizers (of size roughly the third root of the index).  
\end{remark}

\noindent \textbf{Vertex degrees.}\ Since the tree $\T$ is $(q^{\delta}+1)$-valent and $\Gamma \backslash \T^s$ consists of stable vertices (that have trivial stabilizers), we find that all vertices in $\Gamma \backslash \T^s$ are $(q^{\delta}+1)$-valent. In particular, for the maximal vertex degree, we find 
\begin{equation} \label{twee} \Delta_{G} = q^\delta+1. \end{equation}
The boundary vertices $v$ in $\Gamma \backslash (\T^{ss}-\T^s)$ have valency $q^\delta$, since we have already shown that they are connected to a unique cusp.

\medskip

\noindent \textbf{The first eigenvalue of the Laplace operator.}\ 
We will relate the Laplace operator on the finite graph $G:=G_{\Gamma}$ to a Hecke operator on (quotients of) the Bruhat-Tits tree. 

The Hecke operator that we consider is $T_\infty$, associated to the characteristic function of the double coset $$\GL(2,\mathcal{O}_\infty) \left( \begin{smallmatrix} \pi_\infty & 0 \\ 0  & 1 \end{smallmatrix} \right) \GL(2,\mathcal{O}_\infty);$$ equivalently, it acts on the vertices of the Bruhat-Tits tree $\T$ of $\PGL(2,K_\infty)$ as the adjacency operator of $\T$; so $$T_\infty(f)(v) = \sum\limits_{w \colon \{v,w\} \in \E(\sT)} f(w)$$ for a function $f$ on $\V(\T)$. Now $T_\infty$ descends to a Hecke operator on $\Gamma\backslash \T$ by 
$$ T_\infty(f)(v) = \sum\limits_{w \colon \{v,w\} \in \E\Gamma \backslash \sT} \alpha_{vw} \cdot f(w)$$
where we set $\alpha_{vw}:= [\mbox{Stab}_{\Gamma}(v):\mbox{Stab}_{\Gamma}(e)]$ (compare also \cite{Lorscheid}). 
Now consider the adjacency operator $A_G$ of $G=G_\Gamma$, and suppose that $f$ is a function on the vertices of $G$ that is an eigenfunction for $A_G$ with eigenvalue $\lambda$. We claim that it extends uniquely to an eigenfunction $\tilde f$ of $T_\infty$ on the vertices of $\Gamma \backslash \T$. Indeed, on all non-boundary vertices of $G$, $T_\infty = A_G$. Then, if $v$ is a boundary vertex of $A_G$, and $w$ is the unique vertex outside $G$ that is adjacent to $v$, we want $$ \lambda f(v) = T_\infty f(v) = \alpha_{vw} \tilde f(w) + A_G f(v) = \alpha_{vw} \tilde f(w) + \lambda f(v). $$
Hence we should define $\tilde{f}(w):=0$. Finally, if $w_1,w_2,w_3$ are three consecutive vertices outside $G$, then by computing $T_\infty \tilde{f}(w_3)$, we see that we need to define
$$ \tilde{f}(w_3):= \frac{\lambda}{\alpha_{w_2 w_3}} \tilde f(w_2) - \frac{\alpha_{w_1w_2}}{\alpha_{w_2 w_3}} \tilde f(w_1). $$
Now $D=\Gamma \backslash \T$ with the weight function $\alpha_{vw}$ forms a \emph{diagram} in the sense of Morgenstern (\cite{Morgenstern2}, \cite{Morgenstern}), and $\tilde f$ is an eigenfunction for $T_\infty$ in the space $L_2^0(D)$, the orthocomplement of the constant functions in the space of square integrable functions on the vertices for the measure given by the weights $\alpha_{vw}$. As in Theorem 2.1 of \cite{Morgenstern}, the theory of Eisenstein series implies that the continuous spectrum of $T_\infty$ belongs to the segment $[-2 \sqrt{q^\delta}, 2 \sqrt{q^\delta}]$, and Drinfeld's proof of the Ramanujan-Petterson conjecture for function fields (in a series of papers culminating in \cite{Drinfeld}) shows that the discrete eigenvalues $\lambda$ of $T_\infty$ on $L_2^0(D)$ satisfy $|\lambda| \leq 2 \sqrt{q^\delta}$. 
We conclude that $|\lambda| \leq 2 \sqrt{q^\delta}$ holds for the eigenvalues of $A_G$ on $G$. 

The degree matrix of $G$ is given by a diagonal matrix 
$$D_G =  \left( \begin{array}{cc} (q^{\delta}+1)\cdot \mathbf{1} & 0 \\ 0 & q^\delta \cdot \textbf{1} \end{array} \right), $$
where the lower block corresponds to the boundary vertices (which are $q^\delta+1$-valent in $\Gamma \backslash \T$, but only $q^\delta$-valent in $G$, since the cusps are not present in $G$). 

Finally, the Laplacian of $G$ is 
$$ L_G =  L' + B \mbox{ with } L'= q^\delta\mathbf{1} - A_G  \mbox{ and }B=  \left( \begin{array}{cc} \mathbf{1} & 0 \\ 0 & \mathbf{0} \end{array} \right).$$
The Courant-Weyl inequalities (e.g., Theorem 2.1 in \cite{CvDoSa}) imply that 
$ \lambda_G$ is larger than the first eigenvalue of $L'$ (plus the smallest eigenvalue of $B$, which is zero), leading to 
\begin{equation} \label{drie} \lambda_G \geq q^\delta-2 \sqrt{q^\delta}. \end{equation} 

\medskip

\noindent \textbf{Conclusion of the proof of the main bound.}\ Since the function $$\lambda \mapsto \frac{\lambda}{\lambda+4(\Delta+1)}$$ is monotonously increasing in $\lambda$, we find the result by plugging the data from equations (\ref{een}), (\ref{twee}) and (\ref{drie}) in the lower bound from Theorem \ref{spectral}. 

\medskip

\noindent \textbf{Linear lower bound in the genus.}\ We now show how to convert the lower bound on the gonality of $X_\Gamma$ in terms of the index $[\Gamma(Y):\Gamma]$ into a lower bound that is linear in the genus, of the form $$ \mathrm{gon}_{\bar{K}}(X_\Gamma) \geq c'_{K,\delta} (g(X_\Gamma)-1), $$ for $c_K$ a constant depending only on the ground field $K$ and the degree $\delta$ of $\infty$. This is not entirely obvious in positive characteristic, due to wild ramification. 

First of all, it is again enough to establish such a bound for a principal congruence subgroup $\Gamma(Y,\fn)$. First, recall the Riemann-Hurwitz formula for a Galois cover $X \mapsto Y$ with Galois group $G$: 
\begin{equation} \label{RH} 2g_X-2 = |G| \left(2g_Y-2 + \sum_{y \in Y} \sum_{i=0}^\infty \frac{|G_i(y)|-1}{|G_0(y)|} \right), \end{equation}
where $G_i(y)$ are the higher ramification groups of any preimage of $y$ in $X$ (see e.g. \cite{Nakajima}). Applying this to the (Galois) cover (\ref{princov}) and using formula (\ref{prinineq}), it follows that
\begin{align*} \mathrm{gon}_{\bar{K}}(X_\Gamma) & \geq \frac{\mathrm{gon}_{\bar{K}}(X(Y,\fn))}{[\Gamma:\Gamma(Y,\fn)]}|\Gamma\cap Z| \\ & \geq \frac{c'_{K,\delta} (g(X(Y,\fn))-1)}{[\Gamma:\Gamma(Y,\fn)]} |\Gamma\cap Z|\\
& \geq c'_{K,\delta} (g(X_\Gamma)-1+r) \\ & \geq c'_{K,\delta} (g(X_\Gamma)-1), \end{align*}
where we have assumed that the desired bound holds for $\Gamma(Y,\fn)$, and $r \geq 0$ comes from formula (\ref{RH}) applied to the cover (\ref{princov}).  

We now establish the bound for $X(Y,\fn)$. If this curve has genus zero or one, the required bound for the gonality holds trivially. Therefore, we can assume $g(X(Y,\fn)) \geq 2$. The Riemann-Hurwitz formula for the cover $X(Y,\fn) \rightarrow X(Y)$ implies a relation of the form
$$ [\Gamma(Y):\Gamma(Y,\fn)] =  (g(X(Y,\fn))-1) \cdot \frac{2(q-1)}{2g(X(Y))-2 + R}, $$
where $R$ is the term in equation \ref{RH} applied to the Galois cover $X(Y,\fn) \rightarrow X(Y)$.  
Hence to prove our result, it suffices to prove a bound of the form 
$$  2g(X(Y))-2 + R \leq c''_{K,\delta}   $$
for some constant $c''_{K,\delta}$ depending only on $K$ and $\delta$. 

We recall some information about the ``ramification number'' $R$ and the genus $g(X(Y))$ from \cite{GekelerLN} (There, the formulae are worked out for the principal component $Y=A\oplus A$ only, but hold in general). First of all, the genus of $X(Y)$ depends only on $K$ and $\delta$. Secondly, ramification takes place above elliptic points and cusps of $X(Y)$. Let us write $R=R_e+R_c$ with $R_e$ the contribution from elliptic points, and $R_c$ the contribution from cusps. The ramification above elliptic points is tame; and the number of elliptic points depends only on $K$ and $\delta$. Hence $R_e$ is bounded above by a constant in $K$ and $\delta$. 

The ramification above the cusps is wild, but \emph{weak}; this means that the second ramification groups are trivial, and the first ramification group is just the $p$-Sylow group of the stabilizer of the cusp (this follows, for example, from the fact that $X(Y)$ are Mumford curves, hence ordinary---since their Jacobian admits a Tate uniformization, and hence has maximal $p$-rank---, by applying a result of Nakajima \cite{Nakajima}). In the end, we need an upper bound on 
$$ R_c = \frac{q^{d+1}-2}{(q-1)q^d} $$
where $d=\deg(\fn)\geq 1$, that is independent of $d$; for example, $$R_c \leq \frac{q}{q-1}$$ (the limit of $R_c$ as $d$ tends to $+\infty$) will do, and this finishes the proof. 
\end{proof}

\begin{remark}
In the ``standard'' case of a rational function field $K=\F_q(T)$ with a place $\infty$ of degree one, one can make all data explicit. The cover $X(\fn) \rightarrow X(1) \cong \PP^1$ is ramified tamely at the unique elliptic point, of order $q+1$, and at the unique cusp, of order $q^d(q-1)$, where $d=\deg(\fn)$. Hence the Riemann-Hurwitz formula becomes
 \begin{align*} 2 (g(X(\fn))-1) &= [\Gamma(1):\Gamma(\fn)] \left(1-\frac{1}{q+1} -\frac{1}{q^d(q-1)} - \frac{1}{q^d} \right) \\ &\leq [\Gamma(1):\Gamma(\fn)], \end{align*} and it follows that one can set $c'_{K,\delta}=2c_q$ in this case. 
\end{remark}

\begin{remark}
The previous best (non-linear) bounds were due to Schweizer (\cite{SchweizerForum}, Thm.\ 2.4), who showed that if $K$ is a rational function field, then one has, 
for example, $$ \mathrm{gon}_{\bar{\F_q(T)}} X_0(\fn) \geq\frac{1}{\sqrt{(q^2+1)(q+1)}} \cdot [\Gamma(1):\Gamma_0(\fn)]^{\frac{q-1}{2q}}. $$
\end{remark}

\section{Modular degree of elliptic curves over function fields: proof of Theorem \ref{ModDeg}}

\begin{se} Assume that $K$ is a global function field, $\infty$ a place of $K$, and let $E$ denote an elliptic curve over $K$ with split multiplicative reduction at $\infty$ (every non-isotrivial curve acquires such a place of reduction after a finite extension of the ground field $K$). From the work of Drinfeld, it follows that $E$ admits a \emph{modular parametrization} $$\phi \colon X_0(Y,\fn) \rightarrow E$$ (see Gekeler and Reversat \cite{GekelerReversat}) for some suitable modular curves $X_0(Y,\fn)$. This parametrization is defined over the maximal abelian extension $H$ of $K$ that is contained in the completion $K_\infty$. One may study the (minimal) degree of such a modular parametrization, called the \emph{modular degree}. 
\end{se} 

\begin{remark}
Contrary to the case of elliptic curves over $\Q$, in the case where $K=\F_q(T)$, Gekeler has proven that the modular degree always equals the congruence number of the associated automorphic form \cite{GekelerJTNB} \cite{Cojocaru}.  
\end{remark}

\begin{se} We first describe some of the structure of the modular curves $X_0(Y,\fn)$. 
The scheme $M_0(\fn)$, (coarsely) representing the moduli problem of rank-two Drinfeld modules with an $\fn$-isogeny, is defined over $K$, but is not absolutely irreducible if $\mathrm{Pic}(A)$ is non-trivial; it decomposes over $\C_\infty$ as 
 $$ M_0(\fn)(\C_\infty) = \bigsqcup_{Y \in \mathrm{Pic}(A)} \Gamma_0(Y,\fn) \backslash \Omega,$$
where the components are defined over $H$, and sharply transitively permuted by the Galois group $\mathrm{Gal}(H/K) \cong \mathrm{Pic}(A)$. 
One may also describe the modular parametrizations for different $Y$ simultaneously by a $K$-rational map $M_0(\fn) \rightarrow E$, with $M_0(\fn)$ not absolutely irreducible. 
\end{se}

\begin{se}
Since the elliptic curve $E$ admits a map of degree two to $\PP^1$, we find that 
$$ \mathrm{gon}_{\bar{K}}(X_0(Y,\fn)) \leq 2 \deg(\phi).$$
Since we now have a lower bound 
$$ \mathrm{gon}_{\bar{K}}(X_0(Y,\fn)) \geq c_{q,\delta} [\Gamma(Y):\Gamma_0(Y,\fn)], $$
we conclude that 
$$\label{Y01} \deg(\phi) \geq \frac{1}{2} c_{q,\delta} [\Gamma(Y):\Gamma_0(Y,\fn)]. $$
The desired result $\deg \phi \gg_{q,\delta} |\fn|_\infty$ follows from the following lemma. 
\end{se} 

\begin{lemma}
$ [\Gamma(Y):\Gamma_0(Y,\fn)] \geq |\fn|_\infty. $
\end{lemma}

\begin{proof}  
Since both groups $\Gamma(Y)$ and $\Gamma_0(Y,\fn)$ contain the center $Z$, this index is the degree of the covering $X_0(Y,\fn) \rightarrow X(Y). $ Although the different components $X_0(Y,\fn)$ of $M_0(\fn)$ and $X(Y)$ of $M(1)$ depend on $Y$, they are Galois conjugate by $\mathrm{Gal}(H/K) \cong \mathrm{Pic}(A)$. Therefore, the covering \emph{degree} of this cover does not depend on $Y$. Hence we can put $Y=A \oplus A$, and a standard computation then shows that there is a bijection
\begin{align*} \GL(2,A)/\Gamma_0(\fn) & \isomto \PP^1(A/\fn A) \\ \left( \begin{smallmatrix} a & b \\ c & d \end{smallmatrix} \right)  & \mapsto (a:c) \end{align*}
and hence 
$$[\Gamma(Y):\Gamma_0(Y,\fn)] = [\GL(2,A):\Gamma_0(\fn)] =   |\fn |_\infty \cdot \prod_{\p \mid \fn} (1+|\p|_\infty^{-1}) \geq |\fn |_\infty, $$
as was to be proven. 
\end{proof}

\begin{corollary}[= Theorem \ref{ModDeg}] \label{ModDeg2}
Let $E/K$ denote an elliptic curve with split multiplicative reduction at the place $\infty$, of conductor $\fn \cdot \infty$. Then the degree of a modular parametrization $\phi \colon X_0(Y,\fn) \rightarrow E$ is bounded below by 
$$\deg \phi \geq \frac{1}{2} c_{q,\delta} [\Gamma(Y):\Gamma_0(Y,\fn)] \geq \frac{1}{2} c_{q,\delta} |\fn|_\infty. $$
\end{corollary}

\begin{remark}
Previously, Papikian \cite{Papikian} had proven (using Spziro's conjecture for function fields and estimating symmetric square $L$-functions by the Ramanujan conjecture) that for $K=\F_q(T)$ a rational function field, 
$$ \deg_{ns}(j_E) \cdot \deg \phi \gg_{q,\varepsilon}  |\fn|_\infty^{1-\varepsilon},$$ where $j_E$ is the $j$-invariant of $E$ and $\deg_{ns}(j_E)$ is its inseparability degree. He had also proven that $\deg_{ns}(j_E)=1$ if $\fn$ is prime and the curve is \emph{optimal}, i.e., of minimal modular degree in its isogeny class -- a.k.a.\ a strong Weil curve (\cite{PapikianMA}, 1.3). Our lower bound $ \deg \phi \gg_{q,\delta} |\fn|_\infty$ confirms a conjecture that he made in \cite{Papikian}. 
\end{remark}

\begin{remark} \label{pal}
In the other direction, Papikian \cite{Papikiangeneral} has proven an upper bound on the modular degree of an optimal semistable elliptic curve $E$ with square-free conductor for a general function field, depending on the Manin constant $c_E$. Contrary to the case of elliptic curves over $\Q$, one really needs to assume that the curve is optimal, because of the existence of isogenies of arbitrary high degree, arising from the Frobenius operator. 
One should also note that in \cite{Papikiangeneral} the bound is given without the Manin constant as a factor, since it was at first conjectured to always equal one, but P\'al \cite{Pal} has given examples where this is not the case. Also, P\'al has proven a general upper bound for $c_E$ that combines with \cite{Papikiangeneral} to give an upper bound of the form 
$$  \deg(\phi) \ll_{K,\delta} |\fn|^2_\infty \left( \log_q |\fn|_\infty \right)^3   $$
for the degree of an optimal modular cover $\phi$ with square-free conductor $\fn$, and  such that the class number of $K$ is coprime to the characteristic $p$. In particular, the analogue of the \emph{degree conjecture} 
$$ \deg \phi \ll_{q,\varepsilon} |\fn|_\infty^{2+\varepsilon}$$
for any $\varepsilon>0$ holds in this case, and combines with our lower bound. 
\end{remark}

\begin{remark}
Papikian expects that the $j$-invariant of an optimal semi-stable elliptic curve over $K=\F_q(T)$ is separable, and then a lower bound of the form $ \deg \phi \gg c^2_E |\fn|_\infty^{1-\varepsilon}$ can  be shown to hold in many cases, where $c_E$ is the Manin-constant of $E$. The results of P\'al imply that $c_E$ can vary essentially from $1$ to $|\fn|^{1/2}_\infty$, and thus, the value of the Manin constant seems to influence how good our lower bound on the modular degree is. 
\end{remark}

\section{Rational points of higher degree on curves: proof of Theorem \ref{genth}}

We first quote the positive characteristic analogue of a theorem of Frey \cite{Frey}: 

\begin{proposition} \label{FreySchw} Let $X$ denote a curve over a global function field $K$, 
such that its Jacobian does not admit a $\bar{K}$-morphism to a curve defined over a finite field. If $d$ is an integer such that $2d+1 \leq \mathrm{gon}_{\bar{K}}(X)$, then 
the set of points of degree $d$ on $X$ is finite, i.e., 
$$  \left| \bigcup_{[K':K] \leq d} X(K') \right| \leq \infty.  \Box  $$
\end{proposition}

\begin{remark} The result was proven in \cite{SchweizerMZ} (Theorem 2.1) under the assumption that $X$ has a $K$-rational point (similar to a hypothesis of Frey), but Clark has shown that this hypothesis is unnecessary, cf.\ \cite{Clark}, Theorem 5. 
\end{remark}

\begin{remark}
This result has now been improved into a quantitative statement over more general fields by Cadoret and Tamagawa \cite{Cadoret}, as follows: recall that gonality may alternatively be defined as the minimal $d$ for which there exists a non-constant morphism from a $\PP^1$ to the $d$-th symmetric power $X^{(d)}$ of the curve $X$. Define the \emph{isogonality} $\mathrm{isogon}_K(X)$ of $X$ as the minimal $d$ for which there exists a non-constant morphism from a $K$-isotrivial curve to the $d$-th symmetric power $X^{(d)}$ of the curve $X$. Then the result from \cite{Cadoret} says: \emph{for any finitely generated field $K$ of positive characteristic $p>0$, and any smooth geometrically integral curve $X$ over $K$, if $d$ is a natural number with $2d+1 \leq \mathrm{gon}_{\bar{K}}(X)$ and $d+1 \leq \mathrm{isogon}_{\bar{K}}(X)$, then the set of points of degree $\leq d$ on $X$ is finite.}
\end{remark}

\begin{remark} We will typically apply our bound in the following situation: let $K$ be a finitely generated field, and $k$ the fraction field of an excellent discrete valuation ring, with $K \subseteq k$; for example, $K$ is a global function field and $k=K_\infty$ is the completion of $K$ at a place $\infty$; then $\mathrm{gon}_{\bar{K}}(X)=\mathrm{gon}_{\bar{k}}(X)$, so the lower bound that we obtained for $\mathrm{gon}_{\bar{k}}(X)$ from the stable gonality of its reduction graph applies equally well to $\mathrm{gon}_{\bar{K}}(X)$. 
\end{remark}

\begin{remark} If $X/K$ is a Mumford curve over a valued field $k \supseteq K$, then its Jacobian has split reduction, and hence it admits no map to an isotrivial curve, and $\mathrm{isogon}_{\bar{K}}(X)=\gon_{\bar{K}}(X).$ 
\end{remark}

Proposition \ref{FreySchw} and our spectral bound on gonality now immediately imply the following general finiteness result for points on curves whose degree is bounded in terms of spectral data associated to a special fiber: 

\begin{theorem}[=Theorem \ref{genth}] \label{genth2} Let $X$ denote a curve over a global function field $K$,
such that its Jacobian does not admit a $\bar{K}$-morphism to a curve defined over a finite field. Let $K_\infty$ denote the completion of $K$ at a place $\infty$, and let $G$ denote the stable reduction graph of $X/K_\infty$. Let $\Delta$ denote the maximal vertex degree of $G$ and $\lambda$ the smallest non-zero eigenvalue of the Laplacian of $G$. Then the set 
$$  \bigcup_{[K':K] \leq \frac{\lambda(|G|-1)-4\Delta-4}{2\lambda+8\Delta+8}} X(K') $$
of rational points on $X$ of degree at most $\frac{\lambda(|G|-1)-4\Delta-4}{2\lambda+8\Delta+8}$ is finite. $\Box$
\end{theorem}

\begin{example}
Consider the curve $X_{\lambda,r}$ from Example \ref{ASM}, with $\lambda=\lambda(T) \in K:=\F_p(T)$ of negative degree in $T$. Observe that $X_{\lambda,r}$ is a Mumford curve over $k=\F_p((T^{-1}))$. Our gonality bound from Example \ref{ASM} implies for example that if $p^r>14$, then the set of points on $X_{\lambda,r}$ of degree $\leq p^r/6$ is finite. Note that the set of points of degree $p^r+1$ is infinite, so the result is best up to a constant (for varying $p$ and $r$). 
 
\end{example}

\section{Rational points of higher degree on Drinfeld modular curves: proof of Theorem \ref{Tors}}

We now further specialise the results to the case of Drinfeld modular curves: 

\begin{theorem}[= Theorem \ref{Tors}] \label{Tors2} If  $X_\Gamma$ is defined over a finite extension $K_\Gamma$ of $K$, then the set 
$$  \bigcup_{[L:K_\Gamma] \leq  \frac{1}{2} \left( c_{q,\delta} \cdot [\Gamma(1):\Gamma]-1\right)} X_\Gamma(L) $$
is finite. 
\end{theorem}
\begin{proof}
The curves $X_\Gamma$ are Mumford curves for the $\infty$-valuation. Therefore, the conditions to apply Proposition \ref{FreySchw} are satisfied by $X=X_\Gamma$ and $K=K_\Gamma$. 
\end{proof}

\begin{remark}
Since $\Gamma$ is a congruence group, the curve $X_\Gamma$ is covered by some $X(Y,\fn)$, and hence the curve $X_\Gamma$ is defined over $H$. Hence one may always choose $K_\Gamma=H$, but $K_\Gamma$ might be chosen smaller. Also, Drinfeld modular curves always have $H$-rational points, namely, the cusps, so the refinement of result \ref{FreySchw} by Clark is not necessary for this application. 
\end{remark}

We also deduce the following analogue of a result of Kamieny and Mazur \cite{KM}: 

\begin{corollary}[= Theorem \ref{Freyan}] \label{Freyan2}
If $\p$ is a prime ideal in $A$, then the set of all rank two Drinfeld $A$-modules defined over some field extension $L$ of $K$ that satisfies the degree bound $$[LH:H] \leq \frac{1}{2} c_{q,\delta} \cdot |\p|_\infty$$ that admit an $L$-rational $\p$-isogeny is finite.  
\end{corollary}

\begin{proof} Recall that the scheme $M_0(\p)$, coarsely representing this moduli problem, decomposes over $\C_\infty$ as 
 $$ M_0(\p)(\C_\infty) = \bigsqcup_{Y \in \mathrm{Pic}(A)} \Gamma_0(Y,\p) \backslash \Omega,$$
where the components are defined over $H$, and all components have $H$-rational points, namely, the cusps. 

Now a rank-two $A$-Drinfeld module $\phi$ over a field $L$ with an $L$-rational $\p$-isogeny gives rise to an $L$-rational point of $M_0(\p)$, and hence to an $HL$-rational point  $[\phi] \in X_0(Y,\p)(HL)$ for some $Y$. Now the above theorem implies that 
$$ \bigcup_{[HL:H] \leq \frac{1}{2}(\mathrm{gon}_{\bar{H}}(X_0(Y,\p)-1)} X_0(Y,\p)(HL) $$
is finite. Now since by Theorem \ref{DrinGon}, 
$$ \mathrm{gon}_{\bar{H}}(X_0(Y,\p)) \geq c_{q,\delta} [\Gamma(Y):\Gamma_0(Y,\p)], $$
and we have $$[\Gamma(Y):\Gamma_0(Y,\p)]=|\p|_\infty+1, $$
the result follows. 
\end{proof} 

\begin{corollary}[= Corollary  \ref{KM}] \label{KM2}
Fix a prime $\p$ of $A$. There is a uniform bound on the size of the $\p$-primary torsion of any rank two $A$-Drinfeld module over $L$, where $L$ ranges over all extensions for which the degree $[LH:H]$ is bounded by a given constant. 
\end{corollary}

\begin{proof}
The method of proof is similar to the one in Kamieny-Mazur \cite{KM}, as used in \cite{SchweizerMZ}, Thm.\ 2.4: the moduli space $M_0(\p^e)$ has only finitely many $LH$-points as soon as $$e\geq \log_q(2[LH:H]/c_{q,\delta})/\log_q(|\p|_\infty).$$ For each of the finitely many Drinfeld modules $\phi$ over $LH$ corresponding to these points, Breuer \cite{Breuer} has shown that the open adelic image result of Pink and R\"utsche \cite{Pink} implies that the $\p$-primary torsion $\phi[\p^\infty]$ of $\phi$ is bounded by $C[LH:H]$, where $C$ depends on $\phi$, $K$ and $\p$. One may now maximize the bound as $\phi$ runs through these finitely many Drinfeld modules. Also, for any Drinfeld module $\phi$, $$|\phi[\p^{e-1}] | \leq |\p|_\infty^{2(e-1)}.$$ The result follows. 
\end{proof}

\begin{remark}
In general, $[LH:H]$ is bounded from above by $[L:K]$ (with equality if $L$ and $H$ are linearly disjoint). This shows that a bound of the form $[L:K]\leq d$ implies a bound of the form $[LH:H] \leq d$. Hence the uniform boundedness conjecture for rank-two $A$-Drinfeld modules over $K$  \cite{Poonen} follows from the following statement: for fixed $d$,  there are only finitely many $\p$ such that there exists an $L$-rational $\p$-torsion point on an $A$-Drinfeld module over $L$ with $[L:K] \leq d$.
\end{remark}

\appendix
\section{Other notions of gonality from the literature} \label{app}

In this appendix, we describe various other notions of graph gonality from the literature, and discuss the relation of stable gonality to these alternatives. 

\begin{se} We first recall the notion of graph gonality from Caporaso \cite{Caporaso}, but we change the terminology to be compatible with \cite{Amini} and the current paper. For the convenience of the reader, we include a dictionary between the terminology in \cite{Caporaso} and this paper in Table \ref{dict}.  

A \emph{morphism} between two loopless graphs $G$ and $G'$ (denoted by $\f:G\to G'$) is a map $$\varphi \colon \V(G)\cup\E(G)\to\V(G')\cup\E(G')$$ such that  $\varphi(\V(G))\subset \V(G')$, and for every edge $e\in\E(x,y)$, either $\varphi(e)\in\E(\varphi(x),\varphi(y))$ or $\f(e)\in\V(G')$ and $\f(x)=\f(y)=\f(e)$; together with, for every $e\in\E(G)$, a non-negative integer $r_\f(e)$, the \emph{index} of $\f$ at $e$, such that $r_\f(e)=0$ if and only if $\f(e)\in\V(G')$. 

Previously, in Definition \ref{defmor}, we only considered \emph{finite} morphisms, which are morphisms that map edges to edges.  
The notions of harmonicity and degree that we introduced in Definition \ref{defmor} make sense for morphisms, even if they are not finite. A harmonic morphism is called \emph{non-degenerate} if $m_\f(v)\geq 1$ for every $v\in\V(G)$ (this is automatic if it is finite). 
\begin{table}[h] \caption{Small dictionary of terminology} \label{dict}
 \begin{tabular}{ll}
Terminology in \cite{Caporaso} & Terminology in this paper\\
\hline
indexed morphism & morphism\\
homomorphism &finite morphism\\
stable refinement & refinement\\
pseudo-harmonic & harmonic\\
\hline
 \end{tabular}
\end{table}
\end{se}

\begin{se} The \emph{gonality} of a graph is defined to be 
\[\gon(G)=\min\{\deg\f|\f \text{ a non-degenerate harmonic morphism from $G$ to a tree } T\}.\] 
Caporaso proves that \emph{the gonality of a complex nodal curve is bounded below by the gonality of any refinement of its intersection dual graph.} 
\end{se}

\begin{lemma}
The stable gonality of a graph $G$ is equal to the minimum of the gonalities of all its refinements: $$\sgon(G) = \min \{ \gon(G') | G' \text{ is a refinement of } G\}.$$
\end{lemma}

\begin{proof}
It suffices to prove that any non-degenerate harmonic morphism $\varphi \colon G \rightarrow T$ from a graph $G$ to a tree $T$ admits a refinement $\varphi' \colon G' \rightarrow T'$ that is a \emph{finite} harmonic morphism of the same degree as $\varphi$. Thus, let $e=(v_1,v_2) \in G$ denote an edge that is mapped to a vertex $\varphi(e)=x \in \V(T)$. Add an extra leaf $\ell$ to $T$ at $x$ , subdivide $e$ into two edges $(v_1,m)$ and $(m,v_2)$, and map both $e_1$ and $e_2$ to $\ell$. Set $r_{\varphi'}(e_i)=m_\varphi(v_i)$ for $i=1,2$. Finally, add a leaf $\ell_w$ to all $w \in \varphi^{-1}(x)$, map them all to $\ell$, and set $r_{\varphi'}(\ell_w)=m_\varphi(w).$
\end{proof}

The following elementary fact, a ``trivial'' spectral bound on the gonality, does not seem to have been observed before: 

\begin{proposition} \label{conn} 
The gonality of a graph $G$ is bounded below by the edge-connectivity (viz., the number of edges that need to be removed from the graph in order to disconnect it): $$\gon(G) \geq \eta(G).$$ 
If $G$ is a simple graph (i.e., without multiple edges), unequal to a complete graph, then $$\gon(G) \geq \lambda_G.$$
\end{proposition}

\begin{proof} Let $\f \colon G \rightarrow T$ denote a harmonic non-degenerate morphism. Choose any edge $e \in \E(T)$. Since removing $e$ from $T$ disconnects it, $\f^{-1}(e)$ is a set of edges of $G$ whose removal disconnects $G$. Hence 
$$ \gon(G) \geq |\f^{-1}(e)| \geq \eta(G). $$
For a simple graph which is not complete, the bound $$\eta(G) \geq {\lambda_G}$$ is one of the inequalities of Fiedler \cite{Fiedler} (4.1 \& 4.2).
\end{proof}

\begin{remark} The ``trivial'' spectral bound in the above proposition is not very useful in practice, since it does not contain a ``volume'' term (like the Li--Yau inequality). Also, since every graph acquires edge connectivity two or one by refinements, the lower bound in the proposition trivializes under refinements (which are required by the reduction theory of morphisms). 
\end{remark} 

\begin{se} Another notion of gonality of graphs $G$ and, more generally, of \emph{metric} graphs $\Gamma$ was introduced by Baker in \cite{Baker}, defined as the minimal degree $d$ for which there is a $g_d^1$ on $\Gamma$ (in analogy to the definition from algebraic geometry). Following Caporaso, we call this gonality of graphs \emph{divisorial gonality}. In \cite{CaporasoBN}, Caporaso has proven a Brill--Noether upper bound for divisorial gonality. For a fixed unmetrized graph, there is in general no relation (inequality either side) between gonality and divisorial gonality, cf.\ \cite{Caporaso} Remark 2.7 (or the graphs on lines 4 and 7 of Table \ref{table}). 

Since the reduction of a stable curve is naturally a metric graph (cf.\ \cite{Baker}), one should not ignore the metric in connection with gonality of curves. Baker has proven that \emph{the gonality of a curve $X$ is larger than or equal to the divisorial gonality of its metric reduction graph} (\cite{Baker}, Cor.\ 3.2). 

\emph{The stable gonality of a graph is larger than or equal to its \emph{stable} divisorial gonality (i.e., the minimum of the divisorial gonality of all refinements), but the inequality can be strict.} This can be seen from the last line of Table \ref{table}, which shows an example of Luo Ye (taken from \cite{Amini}) of stable gonality $4$ but divisorial gonality $3$ (for the metrization in which all edges have unit length). The stable gonality can be smaller than the divisorial gonality (without refinement), as one can see from the sixth line in Table \ref{table}; but the next line of the table shows a refinement that lowers the divisorial gonality. 
 
The banana graph $B_n$ has divisorial and stable gonality $2$ but edge connectivity $n$ (cf.\ Table \ref{table}), showing that an equality analogous to the one in Proposition \ref{conn} cannot hold for divisorial or stable gonality. Dion Gijswijt remarked that $\mathrm{dgon}(G) \geq \min\{ |G|, \eta(G) \}.$ With Josse van Dobben de Bruyn, he has also proven that the divisorial gonality of a graph is larger than or equal to its treewidth (unpublished, but some preliminaries can be found in \cite{vD}), but the entries in Table \ref{table} show that the inequality can be strict. Lower bounds on treewidth imply such bounds on divisorial gonality (e.g., \cite{Bod}, \cite{SCS}). 

\end{se} 

\begin{se}
It seems that our notion of stable gonality of a graph coincides with the notion of gonality introduced in \cite{Amini} from the viewpoint of tropical geometry. The connection between tropical curves and metric graphs can already be found in Mikhalkin \cite{Mikhalkin}, and the notion of harmonic morphism of metric graphs in Anand \cite{Anand}. 
\end{se}

\begin{se} We have collected some sample values in 
Table \ref{table}. As above, $\lambda_G$ is the first eigenvalue of $L_G$, and $\lambda_G^\sim$ is the first eigenvalue of the normalized Laplacian $L_G^\sim$; $\eta(G)$ is the edge connectivity, $\Delta_G$ the maximal vertex degree, $\vol(G)$ is the volume of the graph, $\mathrm{tw}(G)$ its treewidth; $\gon(G)$ is the gonality, $\mathrm{dgon}(G)$ is the divisorial gonality, and $\mathrm{sgon}(G)$ is the stable gonality of $G$. We leave out the lengthy but elementary calculations (for the divisorial gonality of $K_n$, we refer to \cite{Baker}, 3.3).
\end{se}

{\footnotesize 
\begin{table}[h] \caption{Some graphs and their invariants, including gonalities} \label{table}
        \centering
        \rotatebox{90}{
                \begin{minipage}{\textheight}{
\begin{tabular}{lcccccccccc}
\hline
Graph $G$ &$\mathrm{sgon}(G)$ &$\gon(G)$ & $\mathrm{dgon}(G)$ & $\eta(G)$ & $\mathrm{tw}(G)$ & $\Delta_G$ & $\lambda_G$  & $|G|$ & $\lambda_G^\sim$ & $\vol(G)$ \\
\hline
Complete graph $K_{n}$& $n-1$ & $n-1$& $n-1$ &  $n-1$ & $n-1$ & $n-1$  & $n$   & $n$ & $\frac{n}{n-1}$ & $n(n-1)$ \\
Cycle graph $C_n$& $2$& $2$& $2$ &  $2$  & $2$ & $2$ & $4 \sin^2(\frac{\pi}{n})$   & $n$ &$2 \sin^2(\frac{\pi}{n})$ & $2n$ \\
Utility graph $K_{3,3}$& $3$& $3$ & $3$ & $3$ & $3$ & $3$ & $3$  &   $6$ & $1$ & $18$ \\
Banana graph $B_n$ &$2$ &$n$ & $2$ & $n$ & $2$ & $n$ & $2n$ & $2$ &  $2$ & $2n$ \\ 
\includegraphics[width=17mm]{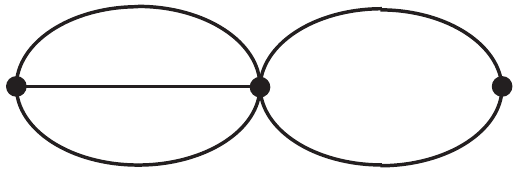} &$3$ & 3& $3$ & $2$ & 2 & $3$ & $5-\sqrt{7} \approx 2.35$&  $3$ &  $\lambda_G$ & $10$ \\ 
\includegraphics[width=17mm]{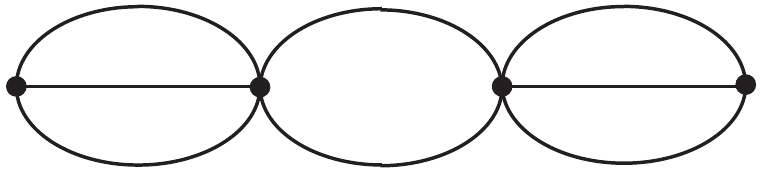}&$3$&$3$&$4$ & $2$ & 2 &  $5$&$5-\sqrt{13}\approx 1.93$&$4$ &$\frac{2}{5}$ & $16$ \\ 
\includegraphics[width=17mm]{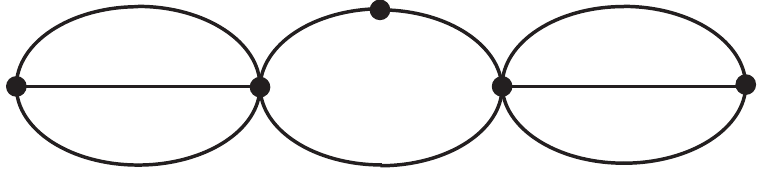}&$3$&?&$3$ & $2$& 2 & $5$ & $\frac{3}{2}(5-\sqrt{13})\approx 1.15$&$5$ &$ \frac{11 - \sqrt{61}}{10} \approx 0.32$ & $18$ \\
\includegraphics[width=15mm]{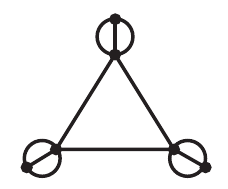}&$4 $&$>3$&$3$&$2$ & 2 & $5$ &$\frac{9-3\sqrt{5}}{2}\approx 1.46$&$6$&$\frac{11-\sqrt{61}}{10} \approx 0.32$&$24$\\ \hline
\end{tabular}
}\end{minipage}}
\end{table}
}

\bibliographystyle{amsplain}

\begin{thebibliography}{10}

\bibitem{ab}
Shreeram Abhyankar, \emph{On the ramification of algebraic functions}, Amer. J.
  Math. \textbf{77} (1955), 575--592.

\bibitem{Abramovich}
Dan Abramovich, \emph{A linear lower bound on the gonality of modular curves},
  Internat. Math. Res. Notices (1996), no.~20, 1005--1011.

\bibitem{AH}
Dan Abramovich and Joe Harris, \emph{Abelian varieties and curves in
  {$W_d(C)$}}, Compositio Math. \textbf{78} (1991), no.~2, 227--238.

\bibitem{Amini}
Omid Amini, Matthew Baker, Erwan Brugall\'e and Joseph Rabinoff, \emph{Lifting harmonic morphisms of tropical curves, metrized complexes, and
Berkovich skeleta}, preprint arxiv:1303.4812 (2013). 


\bibitem{Anand}
Christopher Kumar Anand, \emph{Harmonic morphisms of metric graphs,} in: Harmonic morphisms, harmonic maps, and related topics (Brest, 1997), pp.\ 109--112, Chapman \& Hall/CRC Res.\ Notes Math., vol.\ 413, Boca Raton, FL, 2000. 

\bibitem{Armana}
C\'ecile Armana, \emph{Torsion des modules de {D}rinfeld de rang 2 et formes
  modulaires de {D}rinfeld}, Alg. \& Numb.\ Th. \textbf{6} (2012), 1239--1288.

\bibitem{Baker}
Matthew Baker, \emph{Specialization of linear systems from curves to graphs},
  Algebra Number Theory \textbf{2} (2008), no.~6, 613--653, With an appendix by
  Brian Conrad.

\bibitem{BakerNorine}
Matthew Baker and Serguei Norine, \emph{Harmonic morphisms and hyperelliptic
  graphs}, Int. Math. Res. Not. IMRN (2009), no.~15, 2914--2955.

\bibitem{Bod}
 Hans L. Bodlaender and Arie M.C.A. Koster, \emph{Treewidth computations II. Lower bounds}, Information and Computation \textbf{209} (2011) 1103--1119.

\bibitem{B}
B{\'e}la Bollob{\'a}s, \emph{Random graphs}, second ed., Cambridge Stud.\ in
  Adv.\ Math., vol.~73, Cambridge University Press, Cambridge, 2001.

\bibitem{Breuer}
Florian Breuer, \emph{Torsion bounds for elliptic curves and {D}rinfeld
  modules}, J. Number Theory \textbf{130} (2010), no.~5, 1241--1250.

\bibitem{Cadoret}
Anna Cadoret and Akio Tamagawa, \emph{Points of bounded degree on curves in
  positive characteristic}, in preparation, 2013.

\bibitem{CaporasoBN}
Lucia Caporaso, \emph{Algebraic and combinatorial {B}rill-{N}oether theory},
  Compact moduli spaces and vector bundles, Contemp. Math., vol. 564, Amer.
  Math. Soc., Providence, RI, 2012, pp.~69--85.

\bibitem{Caporaso}
Lucia Caporaso,  \emph{Gonality of algebraic curves and graphs}, preprint, to appear in Klaus Hulek's 60th birthday volume, 
  arXiv:1201.6246, 2012.

\bibitem{Carbone}
Lisa Carbone, Leigh Cobbs, and Scott~H. Murray, \emph{Fundamental domains for
  congruence subgroups of {${\rm SL}_2$} in positive characteristic}, J.
  Algebra \textbf{325} (2011), 431--439.

\bibitem{Chung}
Fan R.~K. Chung, \emph{Spectral graph theory}, CBMS Regional Conf.\ Series in
  Math., vol.~92, Published for CBMS, Washington, DC, 1997.
  
    \bibitem{CLV}
 Fan Chung, Linyuan Lu and Van Vu, \emph{
The spectra of random graphs with given expected degrees}, Internet Math. \textbf{1} (2004), no. 3, 257--275. 
  
  \bibitem{Clark}
  Pete L.~Clark, \emph{On the Hasse principle for Shimura curves}, Israel J. Math. \textbf{171} (2009), 349--365.
  
\bibitem{Cojocaru}
Alina~Carmen Cojocaru and Ernst Kani, \emph{The modular degree and the
  congruence number of a weight 2 cusp form}, Acta Arith. \textbf{114} (2004),
  no.~2, 159--167.

\bibitem{Coleman}
Robert~F. Coleman, \emph{Stable maps of curves}, Doc. Math. (2003), 217--225,
  Extra Vol. (Kazuya Kato's fiftieth birthday).

\bibitem{CKK}
Gunther Cornelissen, Fumiharu Kato, and Aristides Kontogeorgis, \emph{The
  relation between rigid-analytic and algebraic deformation parameters for
  {A}rtin-{S}chreier-{M}umford curves}, Israel J. Math. \textbf{180} (2010),
  345--370.

\bibitem{CvDoSa}
Drago{\v{s}}~M. Cvetkovi{\'c}, Michael Doob, and Horst Sachs, \emph{Spectra of
  graphs}, Pure and Applied Mathematics, vol.~87, Academic Press Inc., New
  York, 1980.

\bibitem{Drinfeld}
Vladimir~G. Drinfeld, \emph{Proof of the Petersson conjecture for $\GL(2)$ over a global field of characteristic $p$},  
Funct.\ Anal.\ Appl.\ \textbf{22} (1988), no. 1, 28--43.  
  
  \bibitem{EM}
Bruce E. Eichinger and J.~E. Martin, \emph{Distribution functions for Gaussian molecules.\ II.\ Reduction of the Kirchhoff matrix for large molecules}, J. Chem. Phys. \textbf{69} (1978), 4595-4599.

\bibitem{Ellenberg}
Jordan~S. Ellenberg, Chris Hall, and Emmanuel Kowalski, \emph{Expander graphs,
  gonality, and variation of {G}alois representations}, Duke Math. J.
  \textbf{161} (2012), no.~7, 1233--1275.
  
    \bibitem{Erdoes}
P.~Erd{\H{o}}s and A.~R{\'e}nyi, \emph{On random graphs. {I}}, Publ. Math.
  Debrecen \textbf{6} (1959), 290--297.


\bibitem{Fiedler}
Miroslav Fiedler, \emph{Algebraic connectivity of graphs}, Czechoslovak Math.
  J. \textbf{23(98)} (1973), 298--305.

\bibitem{Frey}
Gerhard Frey, \emph{Curves with infinitely many points of fixed degree}, Israel
  J. Math. \textbf{85} (1994), no.~1-3, 79--83.

\bibitem{GekelerLN}
Ernst-Ulrich Gekeler, \emph{Drinfeld modular curves}, Lecture Notes in
  Mathematics, vol. 1231, Springer-Verlag, Berlin, 1986.

\bibitem{GekelerJTNB}
Ernst-Ulrich Gekeler,  \emph{Analytical construction of {W}eil curves over function fields},
  J. Th\'eor. Nombres Bordeaux \textbf{7} (1995), no.~1, 27--49.

\bibitem{GekelerNonnengardt}
Ernst-Ulrich Gekeler and Udo Nonnengardt, \emph{Fundamental domains of some
  arithmetic groups over function fields}, Internat. J. Math. \textbf{6}
  (1995), no.~5, 689--708.

\bibitem{GekelerReversat}
Ernst-Ulrich Gekeler and Marc Reversat, \emph{Jacobians of {D}rinfeld modular
  curves}, J. Reine Angew. Math. \textbf{476} (1996), 27--93.

\bibitem{KM}
Sheldon Kamienny and Barry Mazur, \emph{Rational torsion of prime order in
  elliptic curves over number fields}, Ast\'erisque (1995), no.~228, 3,
  81--100.

\bibitem{Kim}
Henry~H. Kim and Peter Sarnak, \emph{Refined estimates towards the {R}amanujan
  and {S}elberg conjectures, \textup{{A}ppendix 2 to: {H}enry {H}.~{K}im,}
  {F}unctoriality for the exterior square of {${\rm GL}_4$} and the symmetric
  fourth of {${\rm GL}_2$}}, J. Amer. Math. Soc. \textbf{16} (2003), no.~1,
  139--183.

\bibitem{Kleiman}
Steven~L. Kleiman and Dan Laksov, \emph{On the existence of special divisors},
  Amer. J. Math. \textbf{94} (1972), 431--436.

\bibitem{LiYau}
Peter Li and Shing~Tung Yau, \emph{A new conformal invariant and its
  applications to the {W}illmore conjecture and the first eigenvalue of compact
  surfaces}, Invent. Math. \textbf{69} (1982), no.~2, 269--291.

\bibitem{Liu}
Qing Liu, \emph{Stable reduction of finite covers of curves}, Compos. Math.
  \textbf{142} (2006), no.~1, 101--118.
  
\bibitem{Lorscheid}
Oliver Lorscheid, \emph{Graphs of Hecke operators}, Algebra Number Theory \textbf{7} (2013), no.\ 1, 19--61.

\bibitem{LiuLorenzini}
Qing Liu and Dino Lorenzini, \emph{Models of curves and finite covers},
  Compositio Math. \textbf{118} (1999), no.~1, 61--102.
  
\bibitem{Mikhalkin}
Grigory Mikhalkin, \emph{Tropical geometry and its applications}, in: Proceedings International Congress of Mathematicians, Vol.\ II, pp.\ 827--852, Eur.\ Math.\ Soc., Z\"urich, 2006. 

  \bibitem{Morgenstern2}
Moshe Morgenstern, \emph{Ramanujan diagrams}, SIAM J.\ Disc.\ Math.\ 
  \textbf{7} (1994), no.~4, 560--570.

\bibitem{Morgenstern}
Moshe Morgenstern, \emph{Natural bounded concentrators}, Combinatorica
  \textbf{15} (1995), no.~1, 111--122.
  
\bibitem{Mumford}
David Mumford, \emph{An analytic construction of degenerating curves over
  complete local rings}, Compositio Math. \textbf{24} (1972), 129--174.

\bibitem{Nakajima}
Sh{\=o}ichi Nakajima, \emph{{$p$}-ranks and automorphism groups of algebraic
  curves}, Trans. Amer. Math. Soc. \textbf{303} (1987), no.~2, 595--607.

\bibitem{Pal}
Ambrus P{\'a}l, \emph{The {M}anin constant of elliptic curves over function
  fields}, Algebra Number Theory \textbf{4} (2010), no.~5, 509--545.

\bibitem{Papikian}
Mihran Papikian, \emph{On the degree of modular parametrizations over function
  fields}, J. Number Theory \textbf{97} (2002), no.~2, 317--349.
  
  \bibitem{PapikianMA}
Mihran Papikian, \emph{Abelian subvarieties of Drinfeld Jacobians and congruences modulo the characteristic}, Math.\ Ann.\ \textbf{337} (2007),139--157.

\bibitem{Papikiangeneral}
Mihran Papikian, \emph{Analogue of the degree conjecture over function fields}, Trans.
  Amer. Math. Soc. \textbf{359} (2007), no.~7, 3483--3503.

\bibitem{Pink}
Richard Pink and Egon R{\"u}tsche, \emph{Adelic openness for {D}rinfeld modules
  in generic characteristic}, J. Number Theory \textbf{129} (2009), no.~4,
  882--907.

\bibitem{PoonenDrinfeld}
Bjorn Poonen, \emph{Torsion in rank {$1$} {D}rinfeld modules and the uniform
  boundedness conjecture}, Math. Ann. \textbf{308} (1997), no.~4, 571--586.

\bibitem{Poonen}
Bjorn Poonen,  \emph{Gonality of modular curves in characteristic {$p$}}, Math. Res.
  Lett. \textbf{14} (2007), no.~4, 691--701.

\bibitem{Rust}
Imke Rust, \emph{Arithmetically defined representations of groups of type
  {${\rm SL}(2,\mathbf F_q)$}}, Finite Fields Appl. \textbf{4} (1998), no.~4,
  283--306.
  
  \bibitem{Saidi}
  Mohamed Sa\"{\i}di, \emph{Rev\^etements \'etales ab\'eliens, courants sur les graphes et r\'eduction semi-stable des courbes}, Manuscripta Math. \textbf{89} (1996), no.~2, 245--265. 

\bibitem{SchweizerMZ}
Andreas Schweizer, \emph{On the uniform boundedness conjecture for {D}rinfeld
  modules}, Math. Z. \textbf{244} (2003), no.~3, 601--614.

\bibitem{SchweizerForum}
Andreas Schweizer, \emph{Torsion of {D}rinfeld modules and gonality}, Forum Math.
  \textbf{16} (2004), no.~6, 925--941.

\bibitem{Serre}
Jean-Pierre Serre, \emph{Trees}, Springer Monographs in Mathematics,
  Springer-Verlag, Berlin, 2003.
  
  \bibitem{SCS}
  L. Sunil Chandran and C.R. Subramanian,  \emph{A spectral lower bound for the treewidth of a graph and its consequences}, Information Processing Letters \textbf{87} (2003), 195--200.

\bibitem{Urakawa}
Hajime Urakawa, \emph{A discrete analogue of the harmonic morphism and {G}reen
   kernel comparison theorems}, Glasg. Math. J., \textbf{42}, (2000), no.~3, 319--334.
   
 \bibitem{vD}
 Josse van Dobben de Bruyn, \emph{Reduced divisors and gonality in finite graphs},
Bachelor thesis, Leiden University, 2012, available at www.math.leidenuniv.nl/scripties/BachvanDobbendeBruyn.pdf.
 
\end{thebibliography}

\end{document}